\documentclass{amsart}
\usepackage[latin9]{inputenc}
\usepackage{color}
\usepackage{amsthm}
\usepackage{amssymb}

\makeatletter
\numberwithin{equation}{section}
\numberwithin{figure}{section}

\usepackage{times}
\usepackage{amsfonts}\usepackage{mathabx}\usepackage{amsrefs}\usepackage{mathrsfs}\usepackage{tikz}\usetikzlibrary{decorations,decorations.pathmorphing}

\newcommand{\N}{{\mathbb N}}
\newcommand{\Z}{{\mathbb Z}}

\newcommand{\fol}{{\rm F{\o}l}}

\newtheorem{thm}{Theorem}[section]

\newtheorem{cor}[thm]{Corollary}
\newtheorem{lem}[thm]{Lemma}
\theoremstyle{definition}
\newtheorem{rem}[thm]{Remark}
\newtheorem{exa}[thm]{Example}

\title[Isoperimetric inequalities]{Isoperimetric inequalities, shapes of F{\o}lner sets and groups with Shalom's property ${H_{\mathrm{FD}}}$}

\author{Anna Erschler}
\thanks{The work of the authors is supported by the ERC grant
GroIsRan. The first named author also thanks the support ANR grant MALIN.
}
\author{Tianyi Zheng}

\date{\today}

\begin{document}
\maketitle

\begin{abstract}
We prove an isoperimetric inequality for groups. As an application we show that any Grigrochuk group of intermediate growth has at least exponential F{\o}lner function. 
As another application, we obtain lower bounds on F{\o}lner functions in various 
nilpotent-by-cyclic groups. Under a regularity assumption, we obtain a characterization of  F{\o}lner functions of 
these groups. As a further application, we evaluate the asymptotics of the F{\o}lner function of 
$Sym(\mathbb{Z})\rtimes {\mathbb{Z}}$.
We study examples of groups with Shalom's property $H_{\mathrm{FD}}$ among nilpotent-by-cyclic groups. We show that there exist
lacunary hyperbolic groups with property $H_{\mathrm{FD}}$. We find groups with property $H_{\mathrm{FD}}$, which are direct products 
of lacunary hyperbolic groups and have arbitrarily large F{\o}lner functions. 
\end{abstract}

\section{Introduction}

Given a finitely generated group $G$, equipped with a symmetric finite generating
set $S$, we denote by $d_{S}$ the word metric on $G$ with respect
to $S$ and $l_S(g)$ the word length $d_{S}(e,g)$.
For a subset $V\subset G$, the boundary $\partial_S V$ of
$V$ is the set of elements of $V$ at the word distance one from
the compliment of $V$: $v\in V$ such that $d_{S}(V,G\setminus V)=1$.
We denote by $F_{G,S}(\epsilon)$ the minimum of the cardinality of
the sets $V$ such that $\#\partial_S V/\#V\le\epsilon$, where the sign $\#V$ denotes the cardinality of $V$. We put
$\fol_{G,S}(n)=F_{G,S}(1/n)$. $\fol_{G,S}(n)$ is called the F{\o}lner function
of $G$. It is well defined whenever $G$ is a finitely generated
amenable group.\\

In this paper we prove an isoperimetric inequality for groups and study its applications.
Our result provides information 
on the structure of F{\o}lner sets: we prove that they contain subsets that are \textquotedblleft{satisfactory\textquotedblright} in
the sense defined below. 	

Given a set $T$ in a group $G$ and a subset $V\subset G$, let us
say that $v\in V$ is $r$-good, if there exists at least $r$ distinct
elements $u\in T$ such that $vu\in V$. Given a constant $C>0$ 
and a finite subset $T\subset G$, we say that $V$ is a {\it $C$-satisfactory
set with respect to $T$} if each $v\in V$ is $C\#T$-good.
We say $T$ is symmetric if $T=T^{-1}$.

\begin{thm} \label{thm:satisfactorysets}

There exist constants $C_1,C_2\in (0,1)$ such that the following holds.
Let $G$ be a group, $S$ be a finite
generating set of $G$. 
Let $T\subset G$ be a symmetric set such that $l_{S}(t)\le r$
for all $t\in T$. Take a subset $V\subset G$ such that $\#\partial_S V/\#V\le C_1/r$.
Then $V$ contains a subset $V'$ which is $C_2$-satisfactory with respect
to $T$.
In particular, we can take $C_1=1/24$ and $C_2=1/4$.

\end{thm}

Given radius $r$, finding an ``optimal'' set $T$ such that every subset $V\subset G$ such that $\#\partial_S V/\#V\le C_1/r$ implies  $V$ contains a subset which is $C_2$-satisfactory with respect to $T$ is related to the {\it dual isoperimetric problem} discussed in \cite[Section 6.3]{gromoventropy}.

Combining Theorem \ref{thm:satisfactorysets} with
Lemma \ref{lem:generalizedLemma3}, which will be proven in Section \ref{inequalities}, we have the following:

\begin{cor} \label{thm:inequality}

Let $G$ be a group with a symmetric finite generating set $S$. Suppose $G$ contains as a subgroup $H=\oplus B_{i}$, $i\in I$, where
$I$ is a countable set. Let 
$v(i,n)=\#\{b\in B_i:\ b\neq e, l_S(b)\le n\}$ and 
$$N(n,k)=\#\{i:\ v(i,n)\ge k\}.$$
Then the F{\o}lner function of $G$ satisfies
$$\fol_{G,S}(n)\ge (k+1)^{CN(n,k)}$$
for any $n,k\in\mathbb{N}$ and some absolute constant $C>0$.

\end{cor}

Corollary \ref{thm:inequality} follows as well from an inequality of Saloff-Coste and the second author \cite[Proposition 3.1]{saloffcostezheng}; and the case with $k=1$ from the inequalities
proven in Gromov \cite[Section 6.1]{gromoventropy} 
where the fibers $B_i$ in the product were assumed to be cyclic assuring that the lower bound is valid not only for the F{\o}lner function, 
but also the linear algebraic F{\o}lner function, introduced in \cite{gromoventropy}.
Before explaining examples where the estimate from Theorem \ref{thm:satisfactorysets} is better than the one obtained from Corollary \ref{thm:inequality}, as in Corollary \ref{cor:sym} below, we point out some consequences of Corollary \ref{thm:inequality}.



Given two functions $f_{1},f_{2}$ we say that $f_{1}\preceq f_{2}$
if $f_{1}(x)\le Cf_{2}(Kx)$ for some positive constants $C,K$ and
all $x$. If $f_{1}\preceq f_{2}$ and $f_{2}\preceq f_{1}$, we say
that $f_{1},f_{2}$ are asymptotically equivalent : $f_{1}\simeq f_{2}$.
It is clear that the asymptotic class of the F{\o}lner function of $G$
does not depend on the choice of a finite generating set $S$ in $G$.
We recall that \textit{the growth function} $v_{G,S}(n)$ counts the
number of elements of word length $\le n$.  The asymptotic class of $v_{G,S}$ does not depend on the choice
of the generating set. We write $\fol_G$ and $v_G$ 
for the equivalence classes of the corresponding functions.

Recall that the wreath product of $A$ and $B$, which we denote by $A\wr B$, is a semi-direct product of $A$ and $\sum_{A} B$, where $A$ acts on $\sum_{A} B$ by shifts (in many papers, the authors denote it by $B\wr A$).
As a particular application of 
Corollary \ref{thm:inequality}, we have that the F{\o}lner function of the wreath product
$G=A\wr B$, where $B$ is a nontrivial finite group satisfies $\fol_{G}(n)\succeq\exp(v_{A}(n))$.
This statement has been known previously since it follows from the inequality of Coulon and Saloff-Coste \cite{CSC} and isoperimetric inequality for wreath product \cite{erschlerfoelner} applied to the case of finite lamp groups. Obtaining this statement as a consequence of Corollary \ref{thm:inequality} provides a unified setting for these inequalities. To see that Corollary \ref{inequalities}
the Coulon-Saloff-Coste inequality 
$
v_A(n)\preceq \fol_A(n)
$,
for a group $A$, take a finite group $B=\Z/2\Z$, and fix a finite generating set $S$ of $A$.
Observe that F{\o}lner function of $G=A\wr B$ with respect to generating set $\tilde{S}=S\cup \{(\delta_{e_A},e_A)\}$ is clearly $\le 2^{\fol_{A,S}(n)}$.
On the other hand, as stated above, we have $\fol_{G,\tilde{S}}(n)\ge\exp(Cv_{A,S}(n))$. We conclude therefore that $\fol_{A,S}(n)\ge v_{A,S}(C''n)$, for some positive
$C''$ and all $n$.\\

Another application of Corollary \ref{thm:inequality} is the following. See the end of Section \ref{inequalities} for the proof of its statement.
\begin{exa}\label{grigorchuk}
The F{\o}lner function of the Grigorchuk group $G_{\omega}$, where $\omega\in \{0,1,2\}^{\mathbb{N}}$ is not eventually constant,  satisfies 
$$\fol_{G_{\omega},S}(n)\ge 2^{Cn}$$
for some absolute constant $C>0$.
\end{exa}
These groups are introduced by Grigorchuk in \cite{grigorchuk85} as first examples of groups of intermediate growth. In particular, the exponential lower bound on the F{\o}lner functions cannot be obtained by the Coulhon-Saloff-Coste inequality mentioned above. It is not known if 
there exists a finitely generated group with intermediate F{\o}lner function, see \cite[Problem 12]{grigorchuksurvey}. The Gap Conjecture of Grigorchuk on F{\o}lner  functions 
(\cite[Conjecture 12.3(ii)]{grigorchuksurvey}) states that the F{\o}lner function of a finitely generated group is either polynomial or at least exponential.
This conjecture is an isoperimetric counterpart of the more well-known Gap Conjecture on volume growth formulated by Grigorchuk in \cite{grigorchukicm}, which asks if a finitely generated group $G$ 
with volume growth function asymptotically strictly smaller than $e^{\sqrt{n}}$ is of polynomial growth. 
The volume growth lower bound $e^{n^{1/2}}$ for $G_{\omega}$, where $\omega$ is not eventually constant, is due to Grigorchuk  (see \cite{grigorchuk85}, 
where it is proven using so called  ``anti-contraction'' property).
A weaker form of the Gap Conjecture asks if there exists a $\beta>0$ such that $v_G(n)=e^{o(n^\beta)}$ implies that $G$ is of polynomial growth.\\




The lower bound obtained from Theorem \ref{thm:satisfactorysets} can be better than
the estimate that one gets from the product subgroups as in Corollary
\ref{thm:inequality}. One of the illustrations is provided in the
Corollary \ref{cor:sym} below.

\begin{cor} \label{cor:sym} Consider a group $Sym(H)$ of finitary
permutations of the elements of a finitely generated group $H$. Let
$G=Sym(H)\rtimes H$ be the extension of this group by $H$. Then
the F{\o}lner function of $G$ is asymptotically greater or equal
than $v_{H}(n)^{v_{H}(n)}$.

 \end{cor}

In particular, the F{\o}lner function of $Sym(\Z)\rtimes\mathbb{Z}$ is asymptotically
equivalent to $n^{n}$. More generally, the F{\o}lner function of
$Sym(\Z^{d})\rtimes\Z^{d}$ (or of $Sym(N)\rtimes N$, $N$ of
growth $\simeq n^{d}$) is asymptotically equivalent to $n^{n^{d}}$.

As another application of the isoperimetric inequalities, we construct nilpotent-by-cyclic groups with
prescribed isoperimetry, see Theorem \ref{prescribe} below.
In particular, for any $a\ge 1$, there exists
a nilpotent-by-cyclic group with
F{\o}lner function asymptotically equivalent to $\exp(n^{a})$. 
Alternatively, we can choose the group to be step-$2$ nilpotent-by-abelian. 
In the special case for $a\in [1, 2]$, the group can be chosen to be step-$2$ nilpotent-by-cyclic.

Gromov states in \cite[Section 8.2, Remark (b)]{gromoventropy} that locally-nilpotent-by-cyclic groups give
examples of elementary amenable groups with prescribed F{\o}lner functions,
provided that the prescribed function has sufficiently fast growing derivatives. 
Further examples of groups with prescribed F{\o}lner functions were provided in 
the paper of Brieussel and the second named author \cite{brieusselzheng}.
In contrast to examples in \cite{brieusselzheng}, 
the groups we consider here are nilpotent-by-cyclic, 
and the regularity assumption on the prescribed function is milder. 
In the special case that the prescribed function is comparable to $\exp(n)$ over infinitely
many sufficiently long intervals, the group in Theorem \ref{prescribe} mentioned above can be chosen to have some additional properties, 
namely to have cautious and diffusive random walk along an infinite time subsequence. Recall that for a random walk $(W_n)$ on $G$,
we say it is {\it diffusive} along a time subsequence if there is a constant $C>0$ such that 
\begin{equation}\label{diffusive}
\liminf_n\mathbb{E} (l_S(W_{n}))/\sqrt{n}\le C.
\end{equation}
Following \cite{erschlerozawa}, we say that a $\mu$-random walk on $G$ is \textit{cautious} if for {\it every} $c>0$,
\begin{equation}
\limsup_{n}\mathbb{P}\left(\max_{1\le k\le n}l_S(W_{k})\le c\sqrt{n}\right)>0.\label{eq: cautious}
\end{equation}
It is not known whether the diffusive condition (\ref{diffusive}) and the cautious condition (\ref{eq: cautious}) are equivalent. 

\begin{thm}\label{prescribe}

Let $\mathfrak{c}\in\N$ and $\tau:\N\to\N$ be a non-decreasing function such that 
\[
0\le \tau(n)\le n^{\mathfrak{c}}.
\]
Then there exists a nilpotent-by-cyclic group $G=N\rtimes\Z$ and a constant $C>1$
such that $N$ is nilpotent with class $\le \mathfrak{c}$ and the F{\o}lner function of $G$ satisfies
\[
Cn\exp\left(n+\tau(n)\right)\ge \fol_{G,S}(n)\ge\exp\left(\frac{1}{C}\left(n+\tau(n/C)\right)\right).
\]
Alternatively, a group satisfying the estimate above can be chosen as a direct product $G_1\times G_2$ where each $G_i$, $i\in\{1,2\}$, 
is a nilpotent-by-cyclic group on which simple random walk is cautious.

\end{thm}

The rest of the Introduction is devoted to discussions on some properties of the examples we construct in Theorem \ref{prescribe} and Section \ref{lacunary} in connection to 
the phenomenon of oscillating F{\o}lner functions and drift functions and Shalom's property $H_{\rm{FD}}$.\\

\subsection{Oscillating F{\o}lner functions and oscillating drift functions}\label{remarks}

Theorem \ref{prescribe} allows us to construct groups with prescribed F{\o}lner functions, in particular, the prescribed F{\o}lner function can oscillate. Recall the drift function of a random walk with step distribution $\mu$ is defined as 
$$L_{\mu}(n)=\sum_{x\in G}d_S(e,x)\mu^{(n)}(x).$$
The phenomenon of oscillating F{\o}lner function and drift function has been observed in various classes of groups. We give a brief review here.

Grigorchuk \cite{grigorchuk85} has constructed groups $G_{\omega}$ of intermediate growth and shown that if
$\omega$ contains long subsequence of constant $000\dots$ then $G_{\omega}$ in some scale is close to a solvable group which is 
a subgroup of $H^m$ ($m$ depends on the scale),
where $H$ is commensurable with a wreath product on which one can check that simple random walk has drift $L_{H,\nu}(n)\le C\sqrt{n}$. 
Thus the results of \cite{grigorchuk85} imply
that for any $\alpha(n)$ tending to infinity, there exists a choice of sequence $\omega$ such that the resulting group 
$G_{\omega}$ is of intermediate volume growth and simple random walk $\mu$ on it satisfies
$$L_{\mu}(n_i) \le \alpha(n_i)\sqrt{n_i}$$ along an infinite subsequence $(n_i)$. 
For the same reason we have that along a subsequence $(m_i)$, $\fol_{G_w}(m_i)\le \exp(m_i\alpha(m_i))$, and one can find a F{\o}lner set $V_i$ with 
$\#\partial V_i/\# V_i\le 1/m_i$ within the ball of radius $m_i \alpha(m_i)$. 
On the other hand, if $\omega$ contains long subsequences of $012012012\dots$, then \cite{grigorchuk85} shows that
$G_w$ is close on some scales to a group commensurable with $G_{012}$. The group $G_{012}$ is defined with sequence 
$\omega=(012)^{\infty}$ and is usually referred to as the first Grigorchuk group.
The result of \cite{erschler-criticalexponent} shows that the drift function of simple random walk on $G_{012}$ satisfies 
$L_{G_{012},\mu}(n)\ge n^{0.65}$ for infinitely many $n$.
This implies that by choosing $\omega$ appropriately, the drift of simple random walk on
$G_w$ can oscillate between $n^{1/2}\alpha(n)$ and $n^{0.65}/\alpha(n)$ for any $\alpha(n)\to \infty $ as $n\to\infty$, 
see \cite[Section 3.1]{erschlerICM}.

The construction of piecewise automatic groups in \cite{erschlerpiecewise} (see also \cite{erschlerICM}) shows that in a group of intermediate growth, the drift function 
$L_{\mu}(n)$ can oscillate
between $n^{1/2} \alpha(n)$ and arbitrarily close to linear; and the F{\o}lner function can be bounded from above by $\exp(n\alpha(n))$ along one subsequence, and arbitrarily large along another subsequence.

In all classes mentioned above it is not known how to evaluate asymptotic of $L_{\mu}(n)$ and $\fol_G(n)$.

The evaluation of $L_{\mu}(n)$ in a family of examples including those with oscillating drift functions was done
by Amir and Virag \cite{amirVirag} , who has shown that a function $f$ 
satisfies the condition that there is a constant $\gamma\in [3/4,1)$ such that $a^{3/4}f(x)\le f(ax)\le a^{\gamma}f(x)$ for all $a>0,x>0$,
then $f$ is equivalent to the drift function of simple random walk on some group.
In a paper of Brieussel and the second author \cite{brieusselzheng}, via a different construction, it was shown that for any function satisfying
$a^{1/2}f(x)\le f(ax)\le af(x)$, there is a group such that $f$ is equivalent to the drift function of simple random walk on it. This construction also provides examples of groups with prescribed F{\o}lner function. 

The examples of $G_{\omega}$ in \cite{grigorchuk85}, piecewise automatic groups in \cite{erschlerpiecewise} and examples of \cite{brieusselzheng} can be chosen to satisfy $L_{\mu}(n) \le \sqrt{n} \alpha (n)$ for infinitely many $n$ for any given $\alpha$ tending to infinity, but none of these examples can have $L_{\mu}(n_i)\le C\sqrt{n_i}$ along some infinite subsequence $(n_i)$. Now our examples discussed in Subsection \ref{pres}, Section \ref{lacunary} and those described in \cite[Subsection 3.3]{brieusselzheng2} provide groups with oscillating drift function satisfying diffusive upper bound along a subsequence (\ref{diffusive}). 
More precisely, examples in Subsection \ref{pres} can be chosen to satisfy the following.

\begin{cor}\label{driftosc}
Given any $\beta\in(\frac{1}{2},1)$,
there exists a nilpotent-by-cyclic group $G$ such that for any symmetric probability measure $\mu$ with finite generating support on $G$,
the drift function $L_{\mu}(n)$ satisfies that $L_{\mu}(t_i)\le Ct_i^{1/2}$ along an infinite subsequence $(t_i)$ and $L_{\mu}(n_i)\ge \frac{1}{C}n_i^{\beta}$ along another infinite subsequence $(n_i)$.
\end{cor}
Corollary \ref{driftosc} is proved in Subsection \ref{osc}.
\begin{subsection}{F{\o}lner functions and groups admitting controlled F{\o}lner pairs along a subsequence}

Theorem \ref{prescribe} provides groups with F{\o}lner function asymptotically equivalent to $e^{n+\tau(n)}$ under some assumptions on $\tau$. Among them some groups have the additional property that they admit controlled F{\o}lner pairs along some subsequences.
In subsection \ref{HFD}, we explain that if a group $G$ admits a subsequence of controlled F{\o}lner pairs, then $G$ has Shalom's property $H_{\rm{FD}}$ (see Corollary \ref{pairfd}). Below we recall necessary background and definitions.

F{\o}lner pairs were introduced in Coulhon, Grigoryan and Pittet \cite{CGP} to produce lower bounds on return probability 
(equivalently upper bounds on $\ell^2$-isoperimetric profiles). 
We recall the definition: 
a sequence $(F'_{n},F_{n})$ of pairs of finite subsets of $G$
with $F'_{n}\subset F_{n}$ is called a sequence of {\it F{\o}lner pairs} adapted to an increasing function $\mathcal{V}(n)$ if 
there is a constant $C<\infty$ and such that for all $n$, we have
\begin{description}
\item[(i)] $\#F_{n}\le C\#F'_{n}$,
\item[(ii)] $d(F'_{n},G\setminus F_{n})\ge n$,
\item[(iii)] $\#F_n\le \mathcal{V}(Cn)$.
\end{description}
Note that admitting such a sequence of F{\o}lner pairs implies that $\fol_G(n)\preceq\mathcal{V}(n)$. 

F{\o}lner pairs provide lower bound for return probability $\mu^{(2n)}(e)$ for a symmetric probability measure $\mu$ of finite support on $G$, see \cite{CGP}. 
We mention that the group with prescribed F{\o}lner function in Theorem \ref{prescribe}
admit F{\o}lner pairs adapted to a functions of the same order as its F{\o}lner function.
For example, from Theorem \ref{prescribe}, we have nilpotent-by-cyclic groups with return probability decay equivalent to $e^{-n^{\alpha}}$, for any $\alpha\in [1/3,\infty)$.

The notion of controlled F{\o}lner pairs was introduced by Tessera in \cite{tessera}.
We say $(F'_{n},F_{n})$ 
with $F'_{n}\subset F_{n}$ is a sequence of {\it controlled F{\o}lner pairs} if they satisfy (i) and (ii) as above, and
\begin{description}
\item[(iii')] $F_n\subset B(e,Cn)$.
\end{description}
Note that it follows from definition that controlled F{\o}lner pairs are adapted to the volume growth function of the group. 
We say that a group admits a subsequence of controlled F{\o}lner pairs if there is an infinite subsequence $(n_i)$ and pairs $(F'_{n_i},F_{n_i})$
which satisfy (i), (ii) and (iii').

Recall that as defined by Shalom in \cite{shalom}, a group $G$ have \textit{property
$H_{\mathrm{FD}}$} 
if every orthogonal $G$-representation $\pi$ with non-zero reduced cohomology $\overline{H}^1(G,\pi)$ admits a finite-dimensional sub-representation. 

By \cite{shalom}, having property $H_{\rm{FD}}$ is invariant under quasi-isometries of amenable groups. Thus it seems natural to find sufficient conditions for property $H_{\rm{FD}}$ which is stable under quasi-isometry as well.
The first sufficient criterion for Property $H_{\rm{FD}}$ which uses random walks is given by Ozawa in \cite{ozawa}. The main application of Ozawa's argument is a new proof of Gromov's polynomial growth theorem. However in \cite[Proposition]{ozawa}, a sufficient criterion for $H_{\rm{FD}}$ is formulated for general groups. In \cite[Corollary 2.5]{erschlerozawa}, it is shown that if $G$ admits a symmetric measure $\mu$ with finite support such that for every $c>0$,
\begin{equation}\label{supball}
\limsup_n \mu^n(B(e,c\sqrt{n}))>0,
\end{equation}
then $G$ has property $H_{\rm{FD}}$.
Admitting a subsequence of controlled F{\o}lner pairs implies (\ref{eq: cautious}), see Lemma \ref{pairs}, therefore  it implies the group has Property $H_{\rm{FD}}$ by \cite[Corollary 2.5]{erschlerozawa}.

The conditions in \cite[Proposition]{ozawa} are not straightforward to to verify even in the simplest examples of exponential growth groups, see \cite{erschlerinvariance}.
The condition of admitting controlled F{\o}lner pairs is more geometric in the sense that
it is stable under quasi-isometry (see \cite{tessera}).
It is not known whether the conditions in \cite[Proposition]{ozawa} and of \cite[Corollary 2.5]{erschlerozawa} are quasi-isometric invariant.

On the other hand, there exist groups with Shalom's property $H_{\rm{FD}}$ such that simple random walks don't satisfy (\ref{supball}), see \cite{brieusselzheng2} and Subsection \ref{HFD}. Such examples show that admitting 
a cautious simple random walk is a strengthening of Property $H_{\rm{FD}}$. 
In particular, we show that groups with property $H_{\rm{FD}}$ where simple random walks 
don't satisfy (\ref{supball}) can be found among nilpotent-by-cyclic groups considered in Theorem \ref{prescribe}. Indeed, from the connection of F{\o}lner function to the decay of return probability, a lower bound on the F{\o}lner function implies an upper bound on the quantity $\mu^n(B(e,r))$, see Remark \ref{non-cautious}. Therefore random walks on groups with F{\o}lner functions growing sufficiently fast do not satisfy (\ref{supball}).

In Section \ref{lacunary}, we construct locally-nilpotent-by-cyclic groups with Shalom's
property $H_{\rm{FD}}$ that are lacunary hyperbolic. 
By definition, lacunary hyperbolic groups are hyperbolic on some scales, 
in some sense they are very \textquotedblleft{large\textquotedblright} on these scales. 
We show that they can be constructed to be at the same time quite \textquotedblleft{small\textquotedblright} 
on some other scales. Namely, they can have 
simple random walk which is cautious and diffusive along some infinite
subsequence.  Among direct products of lacunary hyperbolic groups, there are groups with Property $H_{\rm{FD}}$ whose F{\o}lner function can be arbitrarily large.

There are open questions for some other aspects of lacunary hyperbolic groups, regarding how \textquotedblleft{small\textquotedblright}
they can be on the scales that are not hyperbolic.
For example, Olshanskii, Osin and Sapir \cite{olshosinsapir} ask 
whether a lacunary hyperbolic group can have sub-exponential volume growth or non-uniform exponential growth.
On the other hand, some questions regarding how \textquotedblleft{large\textquotedblright} 
an amenable group with Shalom's property $H_\mathrm{FD}$ can be are also open.  
We mention that in particular, it is not known whether a finitely generated
amenable group admitting a simple random walk with non-trivial Poisson boundary can have property $H_{\mathrm{FD}}$. 
\end{subsection}

\subsection*{Acknowledgements} 
We are grateful to the referee for his/her helpful comments which improved the exposition of the paper. We thank Josh Frisch for pointing out the missing symmetry assumption in Theorem 
\ref{thm:satisfactorysets}. 

\section{Proofs of the isoperimetric inequalities}\label{inequalities}


Let $G$ be a finitely generated group, $S$ be a
finite generating set. Recall that $l_S$ denotes the word distance on $(G,S)$. Take a subset $V\subset G$ such that $\#\partial_S V/\#V\le1/r$.

\begin{lem} [Generalized Coulhon Saloff-Coste inequality]
\label{lem:generalizedCSC}
For each $p:0<p<1$, $C_1>0$ such that $C_1<1-p$ there exists $C>0$ such that the
following holds.  Let $V\subset G$ be a finite subset such that $\#\partial_S V/\#V\le C_1/r$.
Suppose also that $T$ is a subset such that $l_S(t)\le r$ for any $t\in T$.
Then for at least $p\#V$ element $v\in V$ there exists at least
$C\#T$ distinct elements $t\in T$ such that $vt\in V$.
More precisely, we can take $C:1-C=C_1/(1-p)$.

\end{lem}

\begin{proof}

Suppose not. Then for at least $1-p$ proportion of the elements $v$
of $V$, there is strictly less than $C\#T$ multiplications by elements in
$T$ for which $vt\in V$. The idea of the proof is analogous to one
of the known arguments (see e.g. \cite[6.43.]{gromovbook}) for the proof of the Coulhon-Saloff-Coste inequality \cite{CSC}.

For each $v\in V$ consider the sets $U_{v}=\{vs,s\in T\}$, and consider
the total number $N$ of point of the union $U_{v}$ intersected with
$G\setminus V$, taking into account the multiplicity, that is
$$N=\sum_{v\in V}\#(U_v\cap (G\setminus V)).$$
This cardinality
$N$ is at most $\#\partial_S V\#Tr$. Indeed, for each element $t\in T$
fix a geodesics $\gamma_{t}$ from $e$ to $T$. If for some $v\in V$
the element $vt\notin V$, then the shifted geodesic $v,v\gamma_{t}$
intersects the boundary $\partial V$ at least once. Let $u\in \partial V$
be the first intersection of $v,v\gamma_{t}$ with the boundary $\partial V$,
and let $d=d(v,u)$. It is clear that $d\le l_{S}(t)$ and by the
assumption of the lemma we have therefore $d\le r$. Observe that
the point $v$ is uniquely defined by $u$, $t$ and $d$. Therefore the upper bound on $N$
follows. 

On the other hand, by the assumption of the lemma for at least $(1-p)\#V$
points $v$ in $V$ there exist at least $(1-C)\#T$ elements in $T$
such that $vt$ is not in $V$. Therefore, $N\ge(1-p)\#V(1-C)\#T$.
We conclude that 
\[
(1-p)\#V(1-C)\#T< N\le\#\partial V\#Tn\le \frac{C_1}{r}\#V\#Tr,
\]
and hence 
\[
(1-p)(1-C)<C_1,
\]
which contradicts the choice of $C_1$.

\end{proof}

The following combinatorial edge removal lemma is along the same line of reasoning as \cite[Lemma 1]{erschlerfoelner}. 
We consider graphs with non-oriented edges.

\begin{lem}\label{comlemma}
Let $\mathcal{V}$ be a non-empty graph with vertex set $V$ and edge set $E$.
A vertex $v$ is $m$-good if it has at least $m$ edges connecting
to distinct neighbors. Suppose that at least $p\#V$ vertices are
$m$-good, where $p>2/3$, then $\mathcal{V}$ contains a non-empty
subgraph such that every vertex is $(cm)$-good for where $c<\min\{(3p-2)/p,1/2\}$. 
\end{lem}
\begin{proof}
Consider all
the vertices in $V$ that are not $(cm)$-good. Remove all of them
and all the edges adjecent to them. After the removal there can be
new vertices that are not $(cm)$-good. Remove again all of them and
their adjacent edges. Repeat the process. We need to show that the
process stops before the graph becomes empty. 

Orient the removed edges $AB$ as $A\to B$ if $A$ is removed earlier
than $B$. Suppose $v$ is a $m$-good vertex, but gets removed after
several steps. Consider edges adjacent to $v$, then at least $(1-c)m$
edges must be removed in order for $v$ to fail to be $(cm)$-good.
This implies that number of oriented edges ending at $v$ is at least
$(1-c)/c$ as the number of edges starting with $v$. 

Let $C_{1}$ be the number edges in $\mathcal{V}$ such that at least
one of its end vertices is not $m$-good. Let $C_{2}$ be the number
of edges removed that originally both of its end vertices are $m$-good.
Since for every $m$-good vertex that gets removed at some stage,
the number of edges ending at $v$ is at least $(1-c)/c$ as the number
of edges starting with $v$, we have 
\[
\mbox{total number of edges removed}\le C_{1}+C_{2}\le C_{1}+\frac{c}{1-c}C_{1}=\frac{1}{1-c}C_{1}.
\]
Note that $C_{1}\le m(1-p)\#V$. On the other hand, the total amount
of edges in the graph $\mathcal{V}$ satisfies
\[
\left|E\right|\ge\frac{1}{2}p\#Vm.
\]
Since $c<(3p-2)/p$, we have that the total number of edges removed
is strictly less than $|E|$, therefore the remaining graph is non-empty. This finishes the proof the lemma.
\end{proof}

Now we return to the proof of Theorem \ref{thm:satisfactorysets}. 
\begin{proof}[Proof of Theorem \ref{thm:satisfactorysets}]
Given any $p:2/3<p<1$, $0<C_1<1-p$, let $V\subset G$ be a finite subset such that $\#\partial_S V/\#V\le C_1/r$.
Then by Lemma \ref{lem:generalizedCSC}, for at least $p\#V$ elements $v$ in $V$ there exists at least $C\#T$ distinct multiplications by $t\in T$ such that $vt\in V$.
Now consider the graph $\mathcal{V}$ for which the vertex set consists of elements in $V$ and $v,u\in V$ are connected by an edge of there is some $t\in T$ such that
$u=vt$. 
Since $T$ is assumed to be symmetric, $v$ is connected to $u$ if and only if $u$ is connected to $v$, in other words, the graph is non-oriented. 
Then we have that at least $p\#V$ vertices are $C\#T$-good.
By Lemma \ref{comlemma}, $\mathcal{V}$ contains a subgraph $\mathcal{V}'$ such that every vertex is $cC\#T$-good, where $c<\min\{(3p-2)/p,1/2\}$. 
In other words, the vertex set of $\mathcal{V}'$ is $cC$-satisfactory with respect to $T$.
The choice of constants in the statement is somewhat arbitrary: we take $p=5/6$, $C_1=1/24$ and $c=1/3$.  

\end{proof}

In order to derive a lower bound for F{\o}lner function from Theorem \ref{thm:satisfactorysets},
we need to estimate the volume of a set $V'$ which is $C$-satisfactory with respect to $T$.
In the situation that the set $T$ is contained in a product subgroup $\prod B_{i}$, one can bound volume of $V'$
from below by the following lemma. It is analogous to
 \cite[Lemma 3]{erschlerfoelner}

\begin{lem}[Satisfactory sets for $G$ containing $\oplus B_{i}$]
\label{lem:generalizedLemma3} 
Let $G$ be a group containing as a subgroup $H=\oplus B_{i}$, $i\in I$,
$I$ is a countable set. Let $S_{i}$ be a
subset of $B_{i}$. Given a subset $V\subset G$, consider
the cardinality of the following set: for $g\in V$, $k\in\mathbb{N}$,
$$M_k(g)=\#\{i\in I: \mbox{ there exists at least }k \mbox{ distinct } s_i \neq e,s_i\in S_i \mbox{ such that } gs_i \in V \}.$$
If for each $g\in V$,
the cardinality $\#M_k(g)\ge m$, then the cardinality of the set $V$
is at least $(1+k)^{m}$. \end{lem}

\textbf{Proof}. Induction on $m$. The statement is obviously true for $m=1$.
Suppose we have proved the statement
forf $m-1$. Consider the set of coset classes $G=\cup g_{\alpha}H$,
$G_{\alpha}:=g_{\alpha}H$. $V=\cup V_{\alpha}$. If $V$ satisfies
the assumption, all $V_{\alpha}$ satisfy it also. It is enough
therefore to assume that $V$ belongs to one coset class, and without
loss of the generality we can assume that $V\subset H=\prod B_{i}$. 
Fix one position $i$ where the value $B_{i}$ takes at least $k+1$ possible
values. For each fixed value apply the induction hypothesis for $m-1$.
The cardinality of $V$ is at least $(k+1)$ times as much.

\begin{proof}[Proof of Corollary \ref{thm:inequality}]

Let $G$ be a finitely generated group, $S$ is a finite generating
set, $N(n,k)$ defined as the statement. For each index $i\in N(n,k)$, choose $k$ distinct elements 
$\{b_{1,i},\cdots,b_{k,i}\}$ from the set $\{b\in B_i:\ b\neq e, l_S(b)\le n\}$. Define $T$ as the union of these elements
$$T=\bigcup_{i\in N(n,k)}\{b_{1,i},\cdots,b_{k,i}\}.$$
Note that by its definition, $T$ is a symmetric set.

Consider a F{\o}lner set $V\subset G$ such that $\#\partial V/\#V\le1/n$.
Then by Theorem \ref{thm:satisfactorysets}, $V$ contains a subset $V'$
which is $C$-satisfactory with respect to the set $T$ defined above. 
It follows from the definition of $C$-satisfactory and the structure of $T$ that 
for any $g\in V'$,
$$M_{k'}(g)\ge \frac{C}{2-C}N(n,k),\ \mbox{ where } k'=\max\{1,Ck/2\}.$$
By Lemma \ref{lem:generalizedLemma3} the cardinality of such set $V'$
is at least $(k'+1)^{\frac{C}{2-C} N(n,k)}$. Consequently, the cardinality of $V$ admits the same lower bound.
We conclude that the F{\o}lner function
of $G$ satisfies 
$$\fol_{G}(n)\ge (k'+1)^{\frac{C}{2-C} N(n,k)}\ge (k+1)^{C'N(n,k)},$$
where $C,C'$ are some universal constants.
\end{proof}

\begin{proof}[Proof of the statement of Example\ref{grigorchuk}]
We apply Corollary \ref{thm:inequality} to show that the F{\o}lner function of the first Grigorchuk group is at least exponential.

We recall the definition of the first Grigorchuk group $G=G_{012}$, which was defined by Grigorchuk in \cite{grigorchuk80}. The group $G_{012}$ acts on the rooted binary tree $T$, it is generated by automorphisms $a,b,c,d$ defined as follows. The automorphism $a$ permutes the two subtrees of the root. The automorphim $b,c,d$ are defined recursively: 
$$b=(a,c),\ c=(a,d),\ d=(1,b).$$
For more details see \cite{grigorchuk80, grigorchuk85} and also \cite[Chapter VIII]{delaharpe}, \cite[Chapter 1]{BGS}.

The rigid stabilizer of a vertex $u$ in $G$, denoted by ${\rm{Rist}}_G(u)$ is the subgroup of $G$ that consists of these automorphisms that fix all vertices not having $u$ as a prefix. By definition it is clear that automorphisms in rigid stabilizers of different vertices on a given level commute. To obtain a lower bound for the F{\o}lner function, given a level $k$, consider the subgroup $\oplus_{u\in T_k} \rm{Rist}_G(u)$, where $T_k$ denotes the level $k$ vertices in the tree. We show that for in each summand ${\rm{Rist}}_G(u)$, $u\in T_k$, there is an element $g_u\in {\rm{Rist}}_G(u)$ such that $g_u\neq e$ and $l_S(g)\le 6\cdot2^k$. 

Consider the substitution $\sigma$ which was used in the proof of the growth lower bound in \cite[Theorem 3.2]{grigorchuk85}:
$$\sigma(a)=aca,\ \sigma(b)=d,\  \sigma(c)=b,\ \sigma(d)=c.$$
First take $g_{1^k}=\sigma^k(abab)$, where $1^k=\underbrace{1 \cdots 1}_k$. We verify that $g_{1^k}$ is a non-trivial element in the rigid stabilizer of the vertex $1^k$. To show this, it suffices to have 
$$g_{1^k}=(1, \sigma^{k-1}(abab)),$$
where $1$ is the identity element.
By the definition of $\sigma$, we have that for any word $w$ in the letters $\{a,b,c,d\}$, $\sigma(w)$ is in the stabilizer of the first level of the tree, hence we can write 
$\sigma(w)=(w_1,w_2)$.
Observe that $w_2=w$ and 
$w_1$ can be obtained from $w$ by sending $a\to d,\ b\to \emptyset,\ c\to a,\ d\to a$, where $\emptyset$ denote the empty word.
Note that the word $w_1$ is on letters $\{a,d\}$ only. In the first iteration, $\sigma(abab)=acadacad$, by direction calculation we have that 
$$g_1=\sigma(abab)=(1,abab).$$
For $k\ge 1$, $\sigma^k(ab)$ is a word in blocks of $acab$, $acac$ and $acad$. Under the substitution,
$$\sigma(acab)=(dad,acab),\ \sigma(acac)=(dada,acab),\ \sigma(acad)=(dada,acad).$$
Let $w_k$ be the word on the left branch of $\sigma^{k+1}(abab)$ under the wreath recursion. Then $w_k$ is a product of blocks $dad$ and $dada$. Since $a,d$ generate the dihedral group $\langle a,d|a^2=1,d^2=1, (ad)^4=1\rangle$, we have that $(dad)^2=1,\ (dada)^2=1$ and the two elements $dad, dada$ commute.  
Therefore $w_k$ can only evaluate to $dad, dada$ or their product $daddada=a$ in $\langle a,d\rangle$.
In each case we have that $w_kw_k$ evaluates to the identity element. It follows that  
$$\sigma^{k+1}(abab)=(w_kw_k,\sigma^k(ab)\sigma^k(ab))=(1,\sigma^k(abab)).$$
We have proved the claim that $g_{1^{k+1}}=(1, g_{1^k}).$ 
Since $\sigma$ doubles the word length in each iteration, we have that $l_S(g_{1^k})\le 2^{k+2}$.

To get nontrivial elements in ${\rm{Rist}}_u$ for other vertices in level $k$, we take appropriate conjugations of the element $g_{1^k}$. Since the Schreier graph of the level-$k$ vertices with respect to generators $\{a,b,c,d\}$ is connected (and has diameter $2^k$), for any $u\in T_k$, we can fix an element $\rho_u\in G$ such that $u\cdot \rho_u=1^k$ and $l_S(\rho_u)\le 2^k$. Take 
$$g_u=\rho_u g_{1^k} \rho_u^{-1}.$$
Since $g_{1^k}$ is a nontrivial element in ${\rm{Rist}}_G(1^k)$, it is clear that $g_u$ is a nontrivial element in ${\rm{Rist}}_G(u)$.  
We have that $l_S(g_u)\le 2^k+2^{k+2}+2^k=6\cdot 2^k$.
Apply Corollary \ref{thm:inequality} with $H=\oplus_{u\in T_k} {\rm{Rist}}_G(u)$ and $n=6\cdot 2^k$, we have that for some absolute constant $C>0$,
$$\fol_{G,S}(6\cdot 2^k)\ge 2^{C2^k},$$
which is valid for all $k$. In other words, we have an exponential lower bound $\fol_{G,S}(n)\ge 2^{Cn/6}$.

The same argument applies to more general Grigorchuk group $G_{\omega}$, where $\omega$ is not eventually constant. Namely, by taking appropriate substitutions, we can find nontrivial elements in ${\rm{Rist}}_{G_\omega}(u)$ of word length at most $6\cdot 2^k$ for every $u\in T_k$. It follows by Corollary \ref{thm:inequality} that there is an absolute constant $C>0$ such that for any $\omega\in \{0,1,2\}^{\mathbb{N}}$ which is not eventually constant, 
$$\fol_{G_{\omega},S}(n)\ge 2^{Cn}.$$
By \cite{grigorchuk85}, these groups $G_{\omega}$ have sub-exponential volume growth.

\end{proof}


\section{Abelian extensions of nilpotent groups: first examples}\label{step2}

In this section we consider step-$2$-nilpotent-by-cyclic groups, for which we denote by $G=G_{{\rm \Z,Nil,2}}$, $G_{D,2}$, and $G_{D,2,\mathbf{k}}$, see definitions below. Their F{\o}lner functions can be estimated by applying the isoperimetric inequalities in the previous section.

Consider $2$-step free nilpotent group $N=N_{\Z}$,  with
generators $b_{i}$, $i\in\Z$. Let $N_{\Z,2}$ be the quotient
of $N$, over relations $b_{i}^2=1$  for all $i\in \Z$.  
Let $\Z$ act on the
group $N_{\Z,2}$ by shifts of the index set, and consider the extension
group $G=G_{{\rm \Z,Nil,2}}$.

Consider $N$ and $b_i$, $i\in\Z$, as above, put  $b_{i,j}=[b_{i},b_{j}]$, $i>j$. 
Note that $b_{i,j}^2=1$
and $[b_{j},b_{i}]=[b_{i},b_{j}]^{-1}=b_{i,j}^{-1}$.
Given a subset $D\subset\Z$, consider a subgroup $B_{D}$ of the
group generated by $b_{i,j}$, which is generated by $b_{i,j}$ such
that $i-j\in D$. Observe that $B_{D}$ is a normal subgroup of $G_{{\rm \Z, Nil,2}}$.
Indeed, subgroup generated by $b_{i,j}$ is central in the subgroup
generated by $b_{i}$, and the action an element $z\in\Z$ sends $b_{i,j}$
to $b_{i+z,j+z}$, and preserves $i-j$.

Denote by $G_{D,2}$ the quotient group of $G_{{\rm \Z, Nil,2}}$ over $B_{D}$.

Moreover, given a sequence $\mathbf{k}=(k_{j})$ with $\ k_j\in\N$, $j\in \Z$, consider the
quotient of $G=G_{{\rm \Z,Nil,2}}$ over relations 
$$b_{i,i+j}=b_{i+k_{j},i+j+k_{j}}.$$
A subgroup generated by such elements is a normal subgroup, we denote
by $G_{{\rm \Z,Nil,2,\mathbf{k}}}$ the quotient of $G=G_{{\rm \Z,Nil,2}}$
by this normal subgroup and we denote by $G_{D,2,\mathbf{k}}$ the quotient
of $G_{D,2}$ by the image of this subgroup in $G_{D,2}$.

\begin{exa}\label{hall}
When $\mathbf{k}$ is constant $1$, which means $[b_i,b_j]=[b_{i+k},b_{j+k}]$ for all $i,j,k\in\Z$,
the groups $G_{ D,2,\mathbf{1}}$ were considered in P. Hall \cite{phall}. 
Note that in this special case $G_{ D,2,\mathbf{1}}$ is a central extension of $\Z\wr(\Z/2\Z)$.
\end{exa}

The groups $G_{D,2}$ and $G_{D,2,\mathbf{k}}$ admit the following normal forms for elements. 
The group $G_{{\rm \Z, Nil,2}}$ is generated by two element: $b_{0}$ and a generator
$z_{0}$ of $\Z$. We use this generating set $\{b_0,z_0\}$ for its quotients as well.
Any element of $G_{{\rm Nil,2}}$ can be written a $(f,z)$,
where $z\in\Z$ and $f=\prod_{-\infty}^{\infty}b_{i}^{\epsilon_i}\prod_{i<j}b_{i,j}^{\epsilon_{i,j}}$, 
where each $\epsilon_i,\epsilon_{i,j}\in\{0,1\}$. In the product there are only finitely many terms with $\epsilon_i$ or $\epsilon_{i,j}$ nonzero. 
Note that since the elements $b_{i,j}$ are in the center of $N_{\Z,2}$, the ordering of $b_{i,j}$ doesn't matter.
Similarly, any element of $G_{D,2}$ can be written as $(f,z)$, 
where $z\in\Z$,
$$f=\prod_{-\infty}^{\infty}b_{i}^{\epsilon_i}\prod_{i<j,j-i\notin D}b_{i,j}^{\epsilon_{i,j}}.$$
In the further quotient $G_{D,2,\mathbf{k}}$, an element  can be written as $(f,z)$,
where $z\in\Z$,
$$f=\prod_{-\infty}^{\infty}b_{i}^{\epsilon_i}\prod_{i<j,j-i\notin D,0\le i<k_{j-i}}b_{i,j}^{\epsilon_{i,j}}.$$
Given a word in $z_0^{\pm 1},b_0$, the standard commutator collecting procedure (see \cite[Ch.11]{mhall})
rewrites it into the normal form described above.

As an illustration of the isoperimetric inequality in Corollary \ref{thm:inequality},
we first consider the following example.
\begin{exa}
The F{\o}lner function of the group $G_{{\Z,\rm Nil,2}}$ is asymptotically
equivalent to $\exp(n^{2})$.
\end{exa}
\begin{proof}
 Consider the subgroup $B_r\subset B_{\Z}$ generated by $b_{i,j}$,
$-r\le i<j\le r$. The length of each $b_{i,j}$ is at most $C'r$, and $B_r$ is isomorphic to the product of $\langle b_{i,j}\rangle$, $-r\le i<j \le r$.
Applying Corollary \ref{thm:inequality} we conclude that the F{\o}lner
function of $G$ is at least $\exp(Cn^{2})$.

The upper bound that the F{\o}lner function is bounded by $\exp(C'n^2)$ is explained in the proof of (i) of the Corollary \ref{cor:foelnerGDkj} below,
applied by taking $D$ to be the empty set. 
\end{proof}


More generally, Corollary \ref{thm:inequality} provides an optimal
lower bound for the F{\o}lner functions of various quotients of $G_{{\rm Nil,2}}$.


\begin{figure}\label{figure1}
\includegraphics[scale=1]{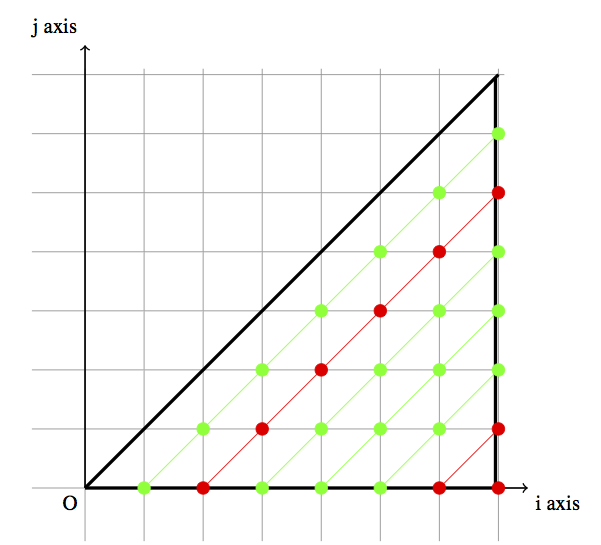}
\caption{Supports of configurations $b_{i,j}$ for $g$ in the F{\o}lner set
$\Omega(n)$ (of the group $G_{\Z,Nil,2}$) are nodes inside the black
triangle; for $\Omega_{D}(n)$ (in the group $G_{D,2}$) these are green
nodes inside the black triangle: green diagonals consist of $i,j,i>j$
such that $i-j\protect\notin D$; this picture is for $n=7$, $D$
contains $2$, $6$ and $7$ and does not contain $1$, $3$, $4$
and $5$}
\end{figure}

\begin{cor}{[}F{\o}lner function of $G_{D,2}$ and $G_{D,2,\mathbf{k}}$.{]}
\label{cor:foelnerGDkj} 

i) Let $\rho_{D}(n)$ be the cardinality
of $\N\setminus D\cap[1,n]$ and $D\subset\N$.
Then the F{\o}lner function of $G_{D,2}$ is asymptotically equivalent
to 
\[
\exp(n\rho_{D}(n)+n).
\]
In particular, for any non-decreasing $\rho(n)$ such that $\rho(n+1)-\rho(n)\le1$
the function $\exp(n\rho(n)+n)$ is a equivalent to a F{\o}lner function
of $G_{D}$, for some choice of a subset $D\subset\N$. 

ii) Let $\tau_{D,\mathbf{k}}(n)=\sum_{j:1\le j\le n,j\notin D}\min\{k_{j},n\}$.
Then the F{\o}lner function of $G_{D,\mathbf{k}}$ satisfies
$$\exp(Cn+C\tau_{D,\mathbf{k}}(Cn))\le \fol_{G_{D,\mathbf{k}}}(n)\le\exp(4n+2\tau_{D,\mathbf{k}}(2n)).$$
In particular, for any non-decreasing function $\tau$ such that $\tau(n+1)-\tau(n)\le n$
the function $n+\tau(n)$ is asymptotically equivalent to 
$\log\fol_{G}(n)$ for some $G=G_{D,2,\mathbf{k}}$.
\end{cor}

\begin{figure}\label{figure2}
\includegraphics[scale=1]{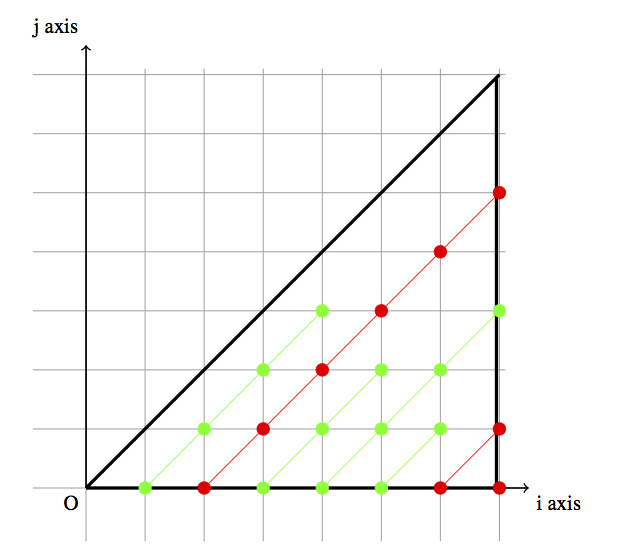}

\caption{Supports of configurations $b_{i,j}$ for $g$ inside the F{\o}lner
set $\Omega_{D,k_{j}}(n)$ (of the group $G_{D,2,\mathbf{k}}$) are green
nodes of truncated diagonals inside the black triangle. The choice
the set of green diagonals (consisting of $i,j:i-j\protect\notin D$)
and the choice of the lengths $k_{i}$ of truncated diagonals is
arbitrary in the definition of $G_{D,2,\mathbf{k}}$; this picture is for
$n=7$, $D$ contains $2$, $6$ and $7$ and does not contain $1$,
$3$, $4$, and $5$, $k_{1}=5$, $k_{3}=4$, $k_{4}\ge5$, $k_{5}=3$}
\end{figure}

\begin{proof}

i). We say that a configuration $f$ is contained in an interval $I$ if in the normal form of $f$, all the nonzero entries $\epsilon_i$, $\epsilon_{i,j}$ satisfy $i,j\in I$.
For $D\subset\Z$ consider the subset $\Omega_{D}(n)$ of $G_{D,2}$,
$$\Omega_{D}(n)=\{(f,z): 0\le z\le n,\mbox{ the configuration }f \mbox{ is contained in }[0,n]\}.$$

Observe that if $g\in\Omega_{D}(n)$, then $zb_{0}\in\Omega_{D}(n)$.
Therefore, if $g\in\partial\Omega_{D}(n)$, then $gz_{0}^{\pm1}\notin\Omega_{D}(n)$
and $z=0$ or $z=n$. For $g\in\Omega_{D}(n)$, consider its \char`\"{}slice\char`\"{}
in $\Omega_{D}(n)$: $h\in \Omega_{D}(n)$ such that $h=fz_{0}^{m}$,
here $z_{0}$ is a generator of $\Z$ and $m$ is some integer. Observe
that for each $g\in\Omega_{D}(n)$, the cardinality of such slice
is equal to $n+1$ and exactly two points of each slice belong to
the boundary f $\partial \Omega_{D}(n)$, for all $n\ge1$.

Therefore for all $n\ge1$

\[
\frac{\#\partial\Omega_{D}(n)}{\#\Omega_{D}(n)}=\frac{2}{n+1}.
\]
Observe that the cardinality of $\Omega_{D}(n)$ is equal to 
\[
(n+1)2^{n+1+\sum_{i\in[0,n]\cap{\N\setminus D}}(n+1-i)}.
\]
It implies that for any $D$, the F{\o}lner function of $G_{D,2}$
is at most $2n\exp(2n+2n\rho_{D}(2n))$.

Let us show that Corollary \ref{thm:inequality} implies that there
exists $C>0$ such that the F{\o}lner function of $G_{D,2}$ is greater
than $\exp(Cn\rho_{D}(Cn))$. 
Consider the collection of $b_{i,j}$ with $0\le i<j\le n$, $j-i\notin D$. The length of $b_{i,j}$ is bounded by
$l_S(b_i,j)\le 4n$, and the subgroup they generate is isomorphic to a product $\prod\langle b_{i,j}\rangle$ where 
in the product there are $\sum_{i\in[0,n]\cap{\N\setminus D}}(n+1-i)$ copies of $\Z/2\Z$.
By Corollary \ref{thm:inequality}, we have that 
$$\fol_{G_{D,2}}(n)\ge \exp\left(C\sum_{i\in[0,n]\cap{\N\setminus D}}(n+1-i) \right)\ge \exp\left(\frac{Cn}{2}\rho_D(n/2)\right).$$
Since $G_{D,2}$ is of exponential growth, we also have $\fol_{G_{D,2}}(n)\ge \exp(Cn)$.
We conclude that 
$$\fol_{G_{D,2}}(n)\sim \exp(n(\rho_D(n)+1)).$$
Finally, given any non-decreasing function $\rho(n)$ with $\rho(n+1)-\rho(n)\le 1$, we can select the set $D\subset \N$ such that
$$\N\setminus D\cap [0,n]\sim \rho(n).$$
Then the corresponding group $G_{D,2}$ has F{\o}lner function as stated. 
\\


The proof of  ii) is similar.
Given a subset $D\subset\N$ and a sequence $\mathbf{k}=(k_{j})$, consider the
function 
\[
\tau_{D,\mathbf{k}}(n)=\sum_{i\in[0,n]\cap(\N\setminus D)}\min\{n,k_{i}-i\}.
\]
With the same argument as in the previous part we see that the
cardinality of $\Omega_{D,\mathbf{k}}$ is equal to 
\[
(n+1)2^{n+1+\sum_{i\in[0,n]\cap{\N\setminus D}}\min\{n+1-i,k_{i}\}},
\]
and that for all $n\ge1$

\[
\frac{\#\partial\Omega_{D,\mathbf{k}}(n)}{\#\Omega_{D,\mathbf{k}}(n)}=\frac{2}{n+1}.
\]

Observe that 
\[
\tau_{D,\mathbf{k}}(n/2)\le\sum_{i\in[0,n]\cap{\N\setminus D}}\min\{n+1-i,k_{i}\}\le \tau_{D,\mathbf{k}}(n).
\]


The F{\o}lner function of $G_{D,2,\mathbf{k}}$ is therefore at most $\exp(n+\tau_{D,\mathbf{k}}(n))$.
By the same argument as in the previous part, from Corollary \ref{thm:inequality} we know that the F{\o}lner
function of $G_{D,2,\mathbf{k}}$ is greater than $\exp(Cn+C\tau_{D,\mathbf{k}}(Cn))$
for some $C>0$.

Given a prescribed non-decreasing function $\tau$ with $\tau(n+1)-\tau(n)\le n$, 
we select $D\subset\N$ and $k_{j}$, $j\in\mathbb{N}$ by the following rule:
if $\lfloor\tau(n)\rfloor=\lfloor\tau(n-1)\rfloor$, then $n\in D$; otherwise $n\notin D$ and 
$k_n=\lfloor\tau(n)\rfloor-\lfloor\tau(n-1)\rfloor$.
Then we have
$$\tau(n/2)\le\tau_{D,\mathbf{k}}(n)\le \tau(n).$$
Then the statement follows.

\end{proof}

\begin{rem}
Under the assumptions of part ii) of Corollary \ref{cor:foelnerGDkj}, we are not able to say that $\fol_{G_{D,\mathbf{k}}}(n)$ is asymptotically equivalent to 
$\exp(n+\tau(n))$ because it is not necessarily true that there exists a constant $A>1$ such that $\tau(An)\ge2\tau(n)$
for all $n$. For example, such constant $A$ doesn't exist for the following function $\tau$: 
take a fast growing sequence $(n_i)$ such that $n_{i+1}$ is much larger than $n_i^2$, 
and take $\tau$ to be equal to $n_i^2$ on the interval $[n_i,n^2_{i}]$ and linear in between such intervals.
\end{rem}

\section{Further examples. Groups with strengthened versions of Property $H_{\rm FD}$}\label{stepc}

For nilpotent groups of higher nilpotency class, 
we can take cyclic extensions similar to the previous section. 
We consider in this section two specific constructions.
In these examples, the group $G$ is a quotient of the semi-direct product $N_{\Z, \rm{Nil},\mathfrak{c}}\rtimes\Z$, where $N=N_{\rm{Nil},\Z,\mathfrak{c}}$ is the step $\mathfrak{c}$ free nilpotent group generated by
$b_i$, $i\in\Z$, subject to relation $b_i^2=1$ for $i\in \Z$. 
The upper bound for the F{\o}lner function of $G$ follows from taking the standard test sets similar to 
these in Corollary \ref{cor:foelnerGDkj}. We apply the isoperimetric inequality in Corollary \ref{thm:inequality} to obtain lower bound 
for the F{\o}lner function: 
we count the rank of the subgroups in the center of the nilpotent group $N$ whose generators are inside the ball of distance $n$ in $G$. 
\subsection{A construction similar to $G_{D,2}$}
We first recall the notion of basic commutators on letters $b_{i}$,
$i\in\mathbb{Z}$, as in \cite[Ch.11]{mhall}. The basic commutators, together
with their weight and an order, are defined recursively as 
\begin{description}
\item [{$1$)}] $c_{i}=b_{i}$, $i\in\mathbb{Z}$ are the basic commutators
of weight $1$, $w(b_{i})=1$; ordered with $b_{i}<b_{j}$ if $i<j$.
\item [{$2$)}] Having defined the basic commutators of weight less than $n$,
the basic commutators of weight $n$ are $[u,v]$ where
$u$ and $v$ are basic with $w(u)+w(v)=n$, $u>v$,
and if $u=[u_1,u_2]$ then $u_2 \le v$.
\item [{$3$)}] For a basic commutator $u$ of weight $n$ and a basic commutator $v$ of weight  $n-1$, we have $u>v$. Commutators of the same weight can be ordered arbitrarily, we use the following: for commutators on different strings, order them lexicographically, and for different bracketing of the same string, order them arbitrarily with respect to each other. 

\end{description}
The basis theorem (see \cite[Theorem 11.2.4]{mhall}) states that the
basic commutators of weight $n$ form a basis for the free abelian
group $F_{n}/F_{n+1}$, where $F_{1}=F$ is the free group on generators
$\left(b_{i}\right)_{i\in\mathbb{Z}}$ and $F_{n}=[F,F_{n-1}]$. 
Observe that the relations $b_i^2=1$ imply that $u^2=1$ for any basic commutator $u$.

Consider the free nilpotent group $N_{\Z,\rm Nil \mathfrak{c}}$ as above and the extension $N_{\mathbb{Z},\rm Nil \mathfrak{c}}\rtimes\mathbb{Z}$,
where $\mathbb{Z}$ acts by translating generators of $N_{\mathbb{Z},\rm Nil \mathfrak{c}}$,
that is $b_{i}^{z}=b_{i+z}$. Note that $\Z$ acts also on commutators by translating the indices, and we write $u^z$ for the conjugation: for $u$ a basic commutator on string $(b_{i_1},...,b_{i_k})$, write $u^z$ for the same bracketing on the string $(b_{i_1+z},...,b_{i_k+z})$.

We will take quotients of the group
$N_{\mathbb{Z},\mathfrak{c}}\rtimes\mathbb{Z}$ in what follows. 
First we set some notations. Let $Q$ be a subset of basic commutators of $N_{\Z,\mathfrak{c}}$ satisfying
the condition that 
\begin{itemize}
\item If $u\in Q$, then $u^z\in Q$ as well for any $z\in\Z$,
\item If $[u,v]\in Q$ where $u$ and $v$ are basic commutators, then $u,v\in Q$ as well. 
\end{itemize}

Let $N_{\mathbb{Z},Q}$ be the quotient of $N_{\mathbb{Z},\mathfrak{c}}$
with the relations that all basic commutators except those in $Q^{\mathbb{Z}}$
are set to be identity. Note that such an operation is well defined
because the basic commutators form a basis in the free nilpotent group
$N_{\mathbb{Z},\mathfrak{c}}$ . 

The commutators in $Q$ inherit the ordering of basic
commutators. The standard word collecting process in $N_{\mathbb{Z},Q}$ (see \cite[Section 11.1]{mhall})
yields a normal form for elements: 
every element $f\in N_{\mathbb{Z},Q}$
can be written uniquely as an ordered product of $b_{j}$'s and basic
commutators in $Q$, that is 
\[
f=\prod_{{c}\in Q}c^{\delta_c}\prod_{j\in\mathbb{Z}}b_{j}^{\varepsilon_{j}},
\]
where $b_j,\delta_c\in\{0,1\}$ and only finitely many $b_j$ and $\delta_c$'s are $1$. 

Let $E_{d}(n)$ be the subset of weight $d$ basic commutators on letters from the set $\{b_{-n},\dots, b_n\}$. 
Write $E(n)=\cup_{d=2}^{\mathfrak{c}}E_{d}(n)$
and $E_d=\cup_{n=1}^{\infty}E_d(n)$.
By construction, basic commutators in $Q\cap E_d$ form a basis for the abelian
quotients $(N_{\mathbb{Z},Q})_{d}/(N_{\mathbb{Z},Q})_{d+1}$, for each $2\le d\le \mathfrak{c}$.

The extensions by $\Z$ of $N_Q$ is denoted by 
$G_{Q}=N_{Q}\rtimes\mathbb{Z}$.

By a similar reasoning as in the step $2$ case, we obtain a lower
bound of the F{\o}lner function by counting the number of basis element
in each abelian quotient.

\begin{cor}

Let $\rho_{Q}(n)$ be the cardinality of the set $Q\cap E(n)$. Then
there exists a constant $C$ depending on $\mathfrak{c}$ such that
the F{\o}lner function of $G_{Q}=N_{\mathbb{Z},Q}\rtimes\mathbb{Z}$
satisfies 
\[
Cn\exp\left(\rho_{Q}(n)\right)\ge \fol_{G_{Q}}(n)\ge\exp\left(\frac{1}{C}\rho_{Q}(n)\right).
\]

\end{cor}

\begin{proof}

We show the lower bounds first. Given $n$, let $d=d(n)\in\{1,\ldots,\mathfrak{c}\}$ be an index
such that $\left|E_{d}(n)\cap Q\right|\ge\frac{1}{\mathfrak{c}}\rho(n)$.
Consider the projection of $G_{Q}$ to step $d$ quotient, so that
now $Q^{\mathbb{Z}}\cap E_{d}$ is in the center of the
nilpotent group. Observe that each basic commutator 
$Q^{\mathbb{Z}}\cap E_d(n)$
is within distance $2^{d+2}n$
to identity element of $G_{Q}$. The number of such non-trivial basis
element is $|E_{d}(n)\cap Q|$. We can now apply Corollary
\ref{thm:inequality} to obtain the lower bound on $\fol_{G_{Q}}(n)$
in the same way as in Corollary \ref{cor:foelnerGDkj}.

The upper bound on the F{\o}lner functions follow from sets $\Omega_{n}$
which is defined as all group elements that can be written as $(f,z)$
where $z\in[-n,n]$ and $f$ is a product of basic commutators only
involving generators $b_{j}$ with $j\in[-n,n]$. 

\end{proof}

\subsection{A construction similar to $G_{D,2,\bf{k}}$ and proof of Theorem \ref{prescribe}}\label{pres}

Now similar to the case when $\Z$ acts on a step-$2$ nilpotent group, we can consider the case that on a nilpotent group of step $\mathfrak{c}$, where $\mathfrak{c}\ge 2$, and impose relations such that the $\mathbb{Z}$-orbits of commutators are periodic.
Note that introducing such relations on commutators has consequences on the relations of higher commutators 
where they appear. Here we will consider a specific way of assigning periodicity, 
which is sufficient to prove Theorem \ref{prescribe}
on nilpotent-by-cyclic groups with prescribed F{\o}lner function.

\begin{proof}[Proof of Theorem \ref{prescribe}]

If $\tau$ is bounded, then the lamplighter group $\mathbb{Z}\wr(\mathbb{Z}/2\mathbb{Z})$
satisfies the inequality in the statement. In what follows we assume
that $\tau$ is unbounded. 

Define a sequence of indices $(k_{i})$ recursively as 
\begin{equation}{ki}\label{ki}
k_{i+1}=\min\{n>k_{i}:\ \tau(2^{n})\ge2\tau(2^{k_{i}})\}.
\end{equation}
By assumption on $\tau$, we have that $\tau\left(2^{k_{i}}\right)\le\left(2^{k_{i}}\right)^{\mathfrak{c}}.$
For each $i$, let $\mathfrak{c}_i$ be the integer such that 
$$\left(2^{k_{i}}\right)^{\mathfrak{c}_i-1}<\tau\left(2^{k_{i}}\right)\le\left(2^{k_{i}}\right)^{\mathfrak{c}_i}.$$
For each $k_{i}$, we consider the group $N_{i}$ which is a quotient group of $N_{\Z,\mathfrak{c}_i}$. The group $N_i$ is generated
by $\{b_{j}, j\in\Z\}$, with relations 
\[
b_{i}^2=1,\ b_{i}=b_{j}\ \mbox{if }i\equiv j\mod2^{k_{i}},
\]
and relations imposing that a commutator on a string $\left(b_{i_{1}},b_{i_{2}},\ldots,b_{i_{\ell}}\right)$
is trivial whenever 
$$\min_{1\le j,k\le\ell}d(\bar{i}_{j},\bar{i}_{k})<2^{k_{i}-\mathfrak{c}_i},$$
where $d$ denotes the distance on the cycle $\mathbb{Z}/2^{k_{i}}\mathbb{Z}$.
Note that the finite cyclic group $\mathbb{Z}/2^{k_{i}}\mathbb{Z}$ acts on $N_i$ by shift. 
Let $G_{i}$ be the extension of $N_{i}$ by $\mathbb{Z}/2^{k_{i}}\mathbb{Z}$.
By construction, the $\Z/2\Z$ rank of the center of $N_i$ satisfies
$$\frac{1}{C}(2^{k_i})^{\mathfrak{c}_i}\le{\rm{Rank}}_{\Z/2\Z}({\rm{Center}}(N_i))\le C (2^{k_i})^{\mathfrak{c}_i}.$$
Under the action of $\Z$, the elements of the center of $N_i$ are divided into $\Z/2^{k_i}\Z$-orbits and each orbit has size $2^{k_i}$.

Further, we choose to keep 
$\lfloor\tau(2^{k_{i}})/2^{k_i}\rfloor$ of distinct $\Z/2^{k_i}\Z$-orbits under the action of $\Z/2^{k_i}Z$ in the center of $N_i$. The other orbits in the center of $N_i$ under action of $\Z/2^{k_i}\Z$ are set to be equal to identity element.
Since this operation is performed in the center, it doesn't affect the lower levels. We denote by $\bar{N}_i$
this quotient group of $N_i$. By construction, the $\Z/2\Z$-rank of the center of $\bar{N}_i$ is comparable to $\tau(2^{k_i})$.

We denote by 
$\bar{G}_{i}$ the extension of $\bar{N}_i$ by $\Z/2^{k_i}\Z$ . It is a quotient group of $\Gamma=N_{\Z,\mathfrak{c}}\rtimes\Z$. Let $\pi_i$ be the quotient map
\[
\pi_{i}:\Gamma\to\bar{G}_{i}.
\]
We take the group $G$ in the claim of the theorem to be 
\[
G=\Gamma/\cap_{i=1}^{\infty}\ker\left(\pi_{i}:\Gamma\to\bar{G}_{i}\right).
\]
That is, $G$ is the smallest group that projects onto each $\bar{G}_i$ marked with generating set $\{b_0,t\}.$\\

We now verify that the F{\o}lner function of $G$ satisfies the estimate as stated. 
In order to prove the lower bound for the F{\o}lner function, for $n\in[2^{k_i},2^{k_{i+1}}]$, consider the quotient group $\bar{G}_i$, 
then we have
$$\fol_{G,S}(n)\ge \fol_{\bar{G}_i,S}(2^{k_i}).$$
By construction of the group $\bar{G}_i$, we have $\tau(2^{k_i})$ central nodes within distance $C2^{k_i}$.
Therefore, by Corollary \ref{thm:inequality} we have 
$$\fol_{\bar{G}_i,S}(C2^{k_i})\ge \exp(\tau(2^{k_i})/C).$$
By definition of $k_{i+1}$ we have that $\tau(n/2)<2\tau(2^{k_i})$, thus we have proved the lower bound on $\fol_{G,S}(n)$.

In order to show the upper bound of the F{\o}lner function, note that for $n\in[2^{k_i},2^{k_{i+1}}]$, the ball of radius $n$ around identity in $\bar{G}_j$ with $2^{k_j-\mathfrak{c}}>n$ are 
the same as the ball of same radius in $\Z\wr(\Z/2\Z)$. Then take the test set to be those elements which have support contained in $[-n,n]$. We have that 
the volume of this set is bounded from above by
$$2n2^{2n}\prod_{j: 2^{k_j-\mathfrak{c}}<n}\#\bar{G}_j\le 2n2^{2n}\prod_{j: 2^{k_j-\mathfrak{c}}<n}\exp(C\tau(2^{k_j}))\le \exp\left(Cn+C\tau(Cn)\right).$$
In the last inequality, we used the fact that by choice of $k_j$, $\tau(2^{k_{j-1}})\le \frac{1}{2}\tau(2^{k_j})$, thus the summation in the exponential function
is bounded by a geometric sum. This completes the proof of the first claim of the theorem. \\

Next we show the second claim of the theorem. For a sequence of increasing integers $(m_i)$, 
and a given prescribed function $\tau$, set 
$$\tau_1(n)=\tau(n)\ \mbox{for }n\in[m_{2j-1},m_{2j}]\mbox{ and }\tau_1(n)=\tau(m_{2j})\mbox{ for } n\in[m_{2j},m_{2j+1});$$
and 
$$\tau_2(n)=\tau(m_{2j-1})\ \mbox{for }n\in(m_{2j-1},m_{2j})\mbox{ and }\tau_2(n)=\tau(n)\mbox{ for } n\in[m_{2j},m_{2j+1}].$$
Then both functions satisfy the assumption that $0\le\tau_i(n)\le n^{\mathfrak{c}}$, $i=1,2$. 
Then as in the proof above, there is a nilpotent-by-cyclic group $G_i$ such that $\log\fol_{G_i}$ is equivalent to $n+\tau_i(n)$. 
It follows that for the direct product $\Gamma=G_1\times G_2$, $\log\fol_{\Gamma}(n)$ is equivalent to 
$n+\max\{\tau_1(n),\tau_2(n)\}=n+\tau(n)$.

It remains to verify that for sufficiently fast growing sequence $(m_i)$, the random walk on each $G_i$, $i=1,2$, is cautious.
Since $\tau_1$ is constant on $[m_{2j},m_{2j+1})$, it follows from the construction that there is index $i(j)$ such that 
$$2^{k_{i(j)}}\le m_{2j}\mbox{ and }\ 2^{k_{i(j)+1}}\ge m_{2j+1}.$$
It implies that in $G_1$, the ball of radius $cm_{2j+1}$ is the same as in 
$$\tilde{G}_{i(j)}=\Gamma/\left(\cap_{1\le k\le i(j)}\ker(\Gamma\to\bar{G}_k)\cap\ker(\pi_0:\Gamma\to \Z\wr(\Z/2\Z))\right).$$
The group $\tilde{G}_{i(j)}$ fits into the exact sequence
$$1\to N_{i(j)}\to\tilde{G}_{i(j)}\to \Z\wr(\Z/2\Z),$$
where $N_{i(j)}$ is a finite nilpotent group depending only on value of $\tau_1$ on $[0,m_{2j}]$. 
Take $m_{2j+1}$ to be sufficiently large such that 
$$m_{2j+1}\gg {\rm{Diam}}_{\bar{G}_{i(j)},S}(N_{i(j)}).$$
Then we have that for simple random walk on $G_1$, $k\le cm_{2j+1}$, 
$$|W_k|\le {\rm{Diam}}_{\bar{G}_{i(j)},S}(N_{i(j)})+|\bar{W}_k|,$$
where $\overline{W}_k$ is the projection of the random walk to $\Z\wr(\Z/2\Z)$.
It follows that for $cm_{2j+1}>t\gg{\rm{Diam}}_{\bar{G}_{i(j)},S}(N_{i(j)})$, for any $c'>0$
$$\mathbb{P}\left(\max_{1\le k\le t}\left|W_{k}\right|_{G_1}\le c'\sqrt{t}\right)\ge\mathbb{P}\left(\max_{1\le k\le t}\left|\overline{W}_{k}\right|_{G_1}\le \frac{c'}{2}\sqrt{t}\right)\ge\delta{c'},$$
where $\delta(c')>0$ is a constant only depending on $c'$. 
The last inequality used the fact that simple random walk on the lamplighter $\Z\wr(\Z/2\Z)$ is cautious.
The argument for the random walk on $G_2$ is the same, by choosing $m_{2j}$ to be sufficiently larger than $m_{2j-1}$. 

\end{proof}

\begin{exa}
A few examples of functions that satisfy the assumption of Theorem \ref{prescribe}:

\begin{itemize}

\item $\tau(n)=n^{\alpha}$, for some $\alpha\in [1,\mathfrak{c}]$.

\item The function $\tau$ admitting an increasing sequence of integers $(n_j)$ and $\alpha,\beta\in (0,\mathfrak{c}]$ such that 
$$\tau(n)=n^{\alpha}\mbox{ for } n\in [n_{2j-1},n_{2j}]\mbox{ and } \tau(n)=n^{\beta}\mbox{ for } n\in [n_{2j},n_{2j+1}].$$
Note that if $\alpha\le 1$, then $n+n^{\alpha}$ is equivalent to $n$. 

\item The function $\tau$ admitting an increasing sequence of integers $(n_j)$ and $\alpha\in (1,\mathfrak{c}]$ such that 
$$\tau(n)=n^{\alpha}\mbox{ for } n\in [n_{2j-1},n_{2j}]$$
and $\tau$ is constant on $[n_{2j},n_{2j+1}]$.

\end{itemize}

\end{exa}

\subsection{Corollaries on return probability and drift}\label{osc}

By the general relation between isoperimetry and decay of return probability
(see e.g. \cite[Theorem 14.3]{woessbook}),
a lower bound on F{\o}lner function implies an upper bound on the return probability 
$\mu^{{n}}(e)$. In particular, if there is some $\delta>0$ such that
the F{\o}lner function of $G$ satisfies $\fol_G(n)\succeq \exp(n^{2+\delta})$ over a sequence of sufficiently long sub-intervals, then there exists a time subsequence $(n_i)$
such that 
$$\mu^{(2n_i)}(e)\preceq\exp\left(-n_i^{\frac{2+\delta}{4+\delta}}\right).$$ 

Recall the drift function of a random walk with step distribution $\mu$ is defined as 
$$L_{\mu}(n)=\sum_{x\in G}d_S(e,x)\mu^{(n)}(x).$$

\begin{rem}\label{non-cautious}
Although in general it is not sharp, one can use the fact that return probability 
$\mu^{(2n)}$ attains its maximum at identity $e$ to give the following bounds:
$$\sum_{g\in B(e,r)}\mu^{(2n)}(g)\le\sum_{g\in B(e,r)}\mu^{(2n)}(e)=v_{G,S}(r)\mu^{(2n)}(e).$$
Since the entropy satisfies $H_{\mu}(2n)\ge -\log\mu^{(2n)}(e)$,
we have 
$$L_{\mu}(2n)\ge cH_{\mu}(2n)\ge -c\log\mu^{(2n)}(e).$$
It follows that in this case if there is some $\delta>0$ such that
the F{\o}lner function of $G$ satisfies $\fol_G(n)\succeq \exp(n^{2+\delta})$ over a sequence of sufficiently long sub-intervals,
the $\mu$-random walk is neither cautious nor diffusive. 
\end{rem}

We now show that groups considered in the previous section with appropriate chosen parameters provide examples with oscillating drift function as stated in Corollary \ref{driftosc} in the Introduction.

\begin{proof}[Proof of Corollary \ref{driftosc}]
Given $\beta$, let $\alpha$ be an exponent such that $\frac{\alpha}{2+\alpha}>\beta$. Take a rapidly growing sequence $(m_i)$, and set
$$\tau(n)=n^{\alpha}\mbox{ for }n\in[m_{2j},m_{2j+1}]\mbox{ and }\tau(n)=m_{2j+1}^{\alpha}\mbox{ for }n\in[m_{2j+1},m_{2j+2}).$$
By Theorem \ref{prescribe}, there is a nilpotent-by-cyclic group $G$ such that 
$$\frac{1}{C}(n+\tau(n))\le\log\fol_{G}(n)\le Cn+C\tau(n).$$
Recall that in the proof of Theorem \ref{prescribe}, the group $G$ is defined as the smallest group that projects onto a sequence $\bar{G}_i$ marked with generating set
 $\{b_0,t\}$, where by construction the Cayley graph of $\bar{G}_i$ coincide 
 with the Cayley graph of the lamplighter $\Z\wr (\Z/2\Z)$ in the ball of radius
 $2^{k_i}$ around identity. The sequence $(k_i)$ is determined from the prescribed function $\tau$ as in (\ref{ki}).

From the definition of the function $\tau$ and the sequence $(k_i)$ as in (\ref{ki}), we have that there exists a subsequence $(p_j)$ such that 
$$k_{p_j+1}/k_{p_j}\ge m_{2j+2}/m_{2j+1}.$$
Indeed, this is because the function $\tau$ is constant on the interval $(m_{2j+1},m_{2j+2})$.
Mark $\bar{G}_0=\Z\wr(\Z/2\Z)$ with the generating set $\{b_0,t\}$ and  write 
\[
G_j=\Gamma/\cap_{i=0}^{j}\ker\left(\pi_{i}:\Gamma\to\bar{G}_{i}\right).
\]
Then the Cayley graph of $G$ with respect to the marking $\{b_0,t\}$ and the Cayley graph of $G_j$ coincides in the balls of radius $2^{k_{j+1}}$ around the identities. 
Note that since each $\bar{G}_i$, $i\ge 1$ is finite, the group $G_j$ is a finite extension of the lamplighter group $\bar{G}_0$. Note that 
$$L_{G,\mu}(n)\le {\rm{Diam}}_{G_j,S}(\ker(G_j\to \bar{G}_0))+L_{\bar{G_0},\mu}(n)+n\mathbb{P}(|W_n|\ge 2^{k_{j+1}}).$$
When $(m_j)$ grows sufficiently fast such that $k_{j+1}$ is much larger than the diameter of the finite set $\ker(G_j\to \bar{G}_0)$ in $G_j$, we have that 
$$L_{G,\mu}(t_i)\le 2L_{\bar{G_0},\mu}(t_i)\le C\sqrt(t_i)$$
along a time subsequence $(t_i)$.

Next we show that if $(m_j)$ grows sufficiently fast, then along another subsequence $(n_i)$,  we have 
$$\mu^{n_i}(e)\le\exp\left(-cn_i^{\frac{\alpha}{2+\alpha}}\right).$$
As explained in Remark \ref{non-cautious}, this implies that the $\mu$-random walk is not cautious and $L_\mu(n_i)\ge cn_i^{\frac{\alpha}{2+\alpha}}$.

We conclude that $L_{\mu}(n)$ oscillates between $n^{\frac{\alpha}{2+\alpha}}$ and $n^{1/2}$, 
\end{proof}

\subsection{Examples of groups with property $H_{\rm{FD}}$}\label{HFD}
As we have mentioned in the Introduction, a group $G$ is said to have \textit{Shalom's property
$H_{\mathrm{FD}}$} if every orthogonal $G$-representation $\pi$ 
with non-zero reduced cohomology $\overline{H}^1(G,\pi)$ admits a finite-dimensional sub-representation. 

By a result of Gournay \cite[Theorem 4.7]{gournay}, if the quotient
of a group over its $FC$-centre has Shalom's property $H_{\mathrm{FD}}$,
then the group also has this property. 
Recall that the conjugacy class of an element $g\in G$ is the set $\{hgh^{-1}: h\in G\}$.
The FC-center of the group consists of elements with finite conjugacy classes. We say $G$ is an FC-central extension of $H$ if the kernel of the quotient map $G\to H$ is contained in the FC-center of $G$. 

\begin{lem}\label{series}
The nilpotent-by-cyclic group $G$ in Theorem \ref{prescribe} is an FC-central extension of $\Z\wr(\Z/2\Z)$.
\end{lem}

\begin{proof}
A formal commutator $u$ on a string $\left(b_{i_{1}},b_{i_{2}},\ldots,b_{i_{\ell}}\right)$ is nontrivial in only finitely many nilpotent groups
$N_i$ in the construction. Indeed,  to be nontrivial in $N_i$, we need $\min_{1\le j,k\le\ell}d(\bar{i}_{j},\bar{i}_{k})\ge 2^{k_{i}-\mathfrak{c}_i}$ on the cycle $\Z/2^{k_i}\Z$.
Since $G=\Gamma/\cap_{i=1}^{\infty}\ker\left(\pi_{i}:\Gamma\to\bar{G}_{i}\right)$, we have that the conjugacy class of image of $u$ in $G$ 
is contained in a subgroup of the direct product
of finitely many $N_i$'s. Since each $N_i$ is a finite nilpotent group, it follows that the commutator $u$ is in the FC-center of $G$. We conclude that 
$G$ is an FC-central extension of $\Z\wr(\Z/2\Z)$.

\end{proof}
Therefore by applying Gournay's result and Lemma \ref{series},
for the groups constructed in Theorem \ref{prescribe}, 
we have that 
they have property $H_{\mathrm{FD}}$. 
Alternatively, by a generalization of Gournay's result, 
\cite[Proposition 4.7]{brieusselzheng} implies that the group $N_{\rm{Nil},\Z,\mathfrak{c}}\rtimes\Z$ has property $H_{\rm{FD}}$. 
Therefore its quotient groups considered in the previous subsection have property $H_{\rm{FD}}$ as well.

Recall that by \cite[Corollary 2.5]{erschlerozawa}, if $G$ admits a cautious random walk $\mu$ with finite generating support,  then $G$ has Property $H_{\rm{FD}}$. A sufficient condition for a symmetric $\mu$-random walk to be cautious is that the $\ell^2$-isoperimetry inside balls satisfies the upper bound 
in the following lemma. Recall that 
$$\lambda_{\mu}(B(e,r))=\inf\left\{\frac{1}{2}\sum_{x,y\in G}(f(x)-f(xy))^2\mu(y):\ \mathrm{supp}f\subseteq B(e,r): \left\Vert f\right\Vert_2=1\right\}.$$
Note that by Coulhon-Saloff-Coste isoperimetric inequality \cite{CSC}, it always admits a lower bound that for some constant $c=c(\mu)>0$, for all $r\ge 1$,
$$\lambda_{\mu}(B(e,r))\ge cr^{-2}.$$
If the opposite inequality holds along a subsequence of balls $B(e,r_i)$, then the random walk is cautious:

\begin{lem}\label{pairs}
Suppose there exists a constant $C>0$
and a sequence of balls $B(e,r_i)$ with $r_i\to \infty$ as $i\to\infty$, 
such that 
\begin{equation}
\lambda_{\mu}(B(e,r_i))\le Cr_i^{-2},\label{eq: iso ball}
\end{equation}
Then the $\mu$-random walk is cautious in the sense of (\ref{eq: cautious}).
\end{lem}

\begin{proof}
Let $(W_k)$ be the $\mu$-random walk trajectory. We first show that 
$$\max_{x\in B(e,r)}\mathbb{P}_x\left(W_k\in B(e,r)\mbox{ for all }k\le n\right)\ge (1-\lambda_{\mu}(e,r))^{n}.$$
This inequality is a known bound which can be obtained from the eigenbasis expansion of the transition probability with Dirichlet boundary
on the ball $B(e,r)$. We provide a proof for the convenience of the reader.
Let $P_n^D$ be the semigroup with Dirichlet boundary on the ball $B(e,r)$.
We recall that by definition
$$P_n^Df(x)=\mathbb{E}_{x}[f(W_n)\mathbf{1}_{\{n<\tau\}}],$$
where $f$ has zero boundary condition on $B(e,r)$, $W_n$ is the random walk with step distribution $\mu$ and $\tau$ is the stopping time that the random walk first exits $B(e,r)$. 
Let $\lambda_1<\cdots\le \lambda_v$ be the eigenvalues of
 $I-\mu$ with Dirichlet boundary on the ball, and $\varphi_1,\cdots,\varphi_v$ be the corresponding eigenfunctions 
 normalized in such a way that $\|\varphi_i\|_2=1$.
 Note that $\lambda_1=\lambda_{\mu}(B(e,r))$ and $\varphi_1$ is non-negative. 
 Then
 $$P_n^D(x,y)=\sum_{i=1}^v(1-\lambda_i)^n\varphi_i(x)\varphi_i(y).$$
 Let $x_0$ be a point in $B(e,r)$ such that $\varphi_1$ achieves its maximum at $x_0$. 
Then
 \begin{align*}
 \sum_{y\in B(e,r)}P_n^D(x_0,y)\varphi_{1}(x_0)\ge  \sum_{y\in B(e,r)}P_n^D(x_0,y)\varphi_{1}(y)\\
=(1-\lambda_1)^n\varphi_{1}(x_0)\left(\sum_{y\in B(e,r)}\varphi_1(y)^2\right)
=(1-\lambda_1)^n\varphi_{1}(x_0).
\end{align*} 
Therefore
 $$\mathbb{P}_{x_0}\left(W_k\in B(e,r)\mbox{ for all }k\le n\right)= \sum_{y\in B(e,r)}P_n^D(x_0,y)\ge (1-\lambda_{\mu}(e,r))^{n}.$$
 
 To see that the assumed bound (\ref{eq: iso ball}) imply cautiousness, note that 
 $$ \mathbb{P}_{e}\left(\max_{k\le n}|W_k|\le 2r\right)\ge \max_{x\in B(e,r)}\mathbb{P}_{x}\left(W_k\in B(e,r)\mbox{ for all }k\le n\right).$$
Given any constant $c>0$, take the time subsequence 
 $n_i=(r_i/c)^2$, we have 
$$
 \mathbb{P}_{e}\left(\max_{k\le n_i}|W_k|\le 2c\sqrt{n_i}\right)\ge (1-Cr_i^{-2})^{r_{i}^2/c^2}\ge \delta(c,C)>0.
$$
\end{proof}

Tessera \cite{tessera} defined a class of groups $(\mathcal{L})$ which includes polycyclic groups, solvable Baumslag-Solitar groups and wreath products $\mathbb{Z}\wr F$
with $F$ finite, and showed that a group from this class admits a full sequence of controlled F{\o}lner pairs.
Moreover, in \cite{tessera2}, he proved this property holds for quotient of any solvable algebraic group over a $q$-adic field where $q$ is a prime. 
Shalom \cite{shalom} proved that polycyclic groups have property $H_{\rm{FD}}$, 
using a theorem of Delorme
\cite{delorme} concerning the cohomology of irreducible unitary representation of connected Lie groups.
Controlled F{\o}lner pairs in polycyclic groups provide another argument to establish property $H_{\mathrm{FD}}$ for these groups, which does not use Delorme's 
result.

Having a controlled F{\o}lner pair $F'_{i}\subset F_{i}$ as above implies (see \cite[Proposition 4.9]{tessera})
$$\lambda_{\mu}(B(e,Cn_i))\le Cn_i^{-2}$$
for any finitely supported symmetric probability measure $\mu$.

\begin{cor}\label{pairfd}
Suppose a finitely generated group admits a subsequence of controlled F{\o}lner pairs, 
then it has property $H_{\mathrm{FD}}$. 
\end{cor}

\begin{rem}
By \cite[Theorem 1]{tessera1}, the asymptotic behavior of the $\ell^2$-isopemetric profile inside balls $\lambda_{\mu}(B(e,r))$ is invariant under quasi-isometry. 
Therefore the property of satisfying the assumption of Lemma \ref{pairs} is stable under quasi-isometry. 
\end{rem}

In Theorem \ref{prescribe} we have shown that group satisfying the claim of the theorem can be chosen to be a direct product of two groups, where each admits a subsequence of controlled F{\o}lner pairs. As a consequence of the discussion above, simple random walk on each factor is cautious.  Thus an alternative way to prove property $H_{\rm{FD}}$ of these  groups is to use \cite[Corollary 2.5]{erschlerozawa} (and the fact that the direct product of groups with Shalom's property also has this property). Corollary \ref{pairfd} is used to show Property $H_{\rm{FD}}$ for certain lacunary hyperbolic groups considered in Section \ref{lacunary} as well. 

It was shown in \cite{brieusselzheng2} that 
the construction of \cite{brieusselzheng}, 
as well as a variation of it can be used
to provide groups with Shalom's property $H_{\mathrm{FD}}$ and prescribed F{\o}lner function. 
Groups we consider in Subsection \ref{step2} and \ref{stepc} provide another simple construction of this kind. These examples show that the property of admitting a cautious simple random walk is strictly stronger than having property $H_{\rm{FD}}$.




\section{Torsion free examples}
We mention here some other extensions of nilpotent groups
where the isoperimetric inequality of Corollary \ref{thm:inequality} can be applied
to obtain good lower bound for the F{\o}lner function. 
In the construction of the groups in Subsections \ref{step2}, \ref{stepc}, we can drop the torsion relation that $b_i^2=1$
and consider cyclic extensions of torsion free nilpotent  groups. 
\begin{exa}\label{ex:torsionfree}
Let $N_{\Z,\mathfrak{c}}$ be the free nilpotent group of class $\mathfrak{c}$ on generators $b_i,i\in\Z$.
Let $G=G_{\Z,\mathfrak{c}}$ be the 
extension $N_{\Z,\mathfrak{c}}\times\Z$, where $\Z$ acts by shifting indices. Then $S=\{b_0,t\}$ is a generating set of $G$. 
The F{\o}lner function of $G$ is asymptotically equivalent to $n^{n^{\mathfrak{c}}}$. 
\begin{proof}
The proof of the lower bound is similar to the torsion case, we look for elements in the center of $N_{\Z,\mathfrak{c}}$ that can be reached within distance $n$ 
and apply Corollary \ref{thm:inequality}.

For a given tuple $i_1<i_2<\dots<i_{\mathfrak{c}}$, let $u(i_1,\dots,i_{\mathfrak{c}})=[[b_{i_1},b_{i_2}],\dots, b_{i_{\mathfrak{c}}}]$. 
Note that it is in the center of $N$ and belongs to the collection of basic commutators. We have that 
$$u(i_1,\dots,i_{\mathfrak{c}})^m=[[b_{i_1},b_{i_2}],\dots, b_{i_{\mathfrak{c}}}^m].$$
Given $n$, consider the abelian subgroup of the center of $N$ 
$$H_n=\prod_{0< i_1<\dots<i_{\mathfrak{c}}}\langle u(i_1,\dots,u_{\mathfrak{c}})\rangle.$$
For each cyclic factor in $H_n$, we have the length estimate 
$$l_S(u(i_1,\dots,u_{\mathfrak{c}})^m)=l_S([[b_{i_1},b_{i_2}],\dots, b_{i_{\mathfrak{c}}}^m])\le C(\mathfrak{c})(n+m),$$
where $C(\mathfrak{c})>0$ is a constant only depending on $\mathfrak{c}$.
Now we apply Corollary \ref{thm:inequality} with the choice of length bound $2C(\mathfrak{c})$ and $k=n$. 
The number $N(2C(\mathfrak{c})n,n)$ of indices such that that are least least $n$ distinct non-identity elements in the cyclic group 
$\langle u(i_1,\dots,u_{\mathfrak{c}})\rangle$ is $\binom{n}{\mathfrak{c}}$. 
It follows from Corollary \ref{thm:inequality} that 
$$\fol_{G,S}(2C(\mathfrak{c})n)\ge n^{Cn^{\mathfrak{c}}}.$$
The upper bound that $\fol_{G,S}(n)\le n^{C'(\mathfrak{c})n^{\mathfrak{c}}}$ follows from choosing the following test set. For an element in $G$, it can be written uniquely 
in the normal form $(f,z)$ where $z\in Z$ and $f$ is a ordered product in terms of basic commutators $f=\prod u^{f(u)}$, where $f(u)\in\Z$ is non-zero for only finitely many commutators. Recall that $w(u)$ denotes the weight of $u$, which
is equal to the commutator length, e.g. $w([i_1,i_2])=2$. 
Take the subset 
$$\Omega_{n}=\{(f,z)\in G:\ {\rm{supp}}f\subset[0,n], |f(u)|\le n^{w(u)},\ 0<z\le n\},$$
it is easy to verify that $\#\Omega_n/\Omega_n\le C/n$. And it is clear that the volume of $\Omega_n$ is bounded by
$$\#\Omega_n\le n\cdot n^{C(\mathfrak{c})n^{\mathfrak{c}}}.$$
 
\end{proof}

\end{exa}

We can then consider various quotients of $G_{\Z,\mathfrak{c}}$ similar to Subsection \ref{stepc}, namely by adding relations that certain basic commutators vanish or 
imposing finite orbits under the action of $\Z$ as in the proof of Theorem \ref{prescribe}. The resulting groups are torsion-free-nilpotent by cyclic. To bound the F{\o}lner function from below,
in the same way as illustrated in Example \ref{ex:torsionfree}, we look for elements in cyclic factors of the center of the nilpotent group within the ball of radius $n$ of identity in the ambient group and apply Corollary \ref{thm:inequality}. 

Since torsion free nilpotent groups are left-orderable,
their cyclic extensions are left-orderable as well, see e.g. \cite[subsection 2.1.1]{GOD}. By a result of Gromov \cite[Section 3.2]{gromoventropy}, for such left orderable groups, the linear algebraic F{\o}lner function (for definition see \cite{gromoventropy}) is equal to  the usual (combinatorial) F{\o}lner function.
Removing the torsion relations from the construction of groups in Theorem \ref{prescribe}, 
we obtain that for any prescribed increasing function $0<\tau(n)\le n^{\mathfrak{c}}\log n$, 
there exists a group $G=N\rtimes \Z$ where $N$ is torsion free nilpotent of step $\le \mathfrak{c}$ and $\log\fol_{G,S}(n)$
is asymptotically equivalent to $n\log n+\tau(n)$.
Moreover in the lower central series of $N$, each quotient is torsion-free abelian. 
Therefore the group $G$ is left-orderable.
It follows that the linear algebraic F{\o}lner function coincide 
with the usual combinatorial F{\o}lner function, thus satisfying the same estimate.

For the free nilpotent group $N_{{\Z}^d,\mathfrak{c}}$ of step $\mathfrak{c}$ on generators $b_x,x\in {\Z}^d$, we can consider
the extension $N_{{\Z}^d,\mathfrak{c}}\rtimes {\Z}^d$ and its various quotients. 
The method of looking for elements in cyclic factors in the center of the nilpotent group and applying
Corollary \ref{thm:inequality} provides sharp lower bounds for the F{\o}lner functions of these groups as well. 
For example, we have
$$\fol_{N_{{\Z}^d,\mathfrak{c}}\rtimes {\Z}^d}(n)\succeq n^{n^{d\mathfrak{c}}}.$$

We mention that none of the torsion free nilpotent-by-abelian groups discussed in this section have Shalom's property $H_{\rm{FD}}$, because they all admit 
$\Z\wr\Z$ as a quotient group.
The wreath product $\Z\wr\Z$ doesn't have property $H_{\rm{FD}}$ by \cite[Theorem 5.4.1.]{shalom}.

\section{Extensions of a symmetric group on a countable set}

Let $H$ be a finitely generated group equipped with a symmetric finite generating set $S$, e.g. $H=\Z^{d}$. 
Consider the group of permutations of $H$ with finite support. Let $Sym_{H}$ be the extension
of this group by $H$. It is clear that $Sym_{H}$ is a finitely generated
group, one choice of generators $\tilde{S}$ is transpositions $(e,s)$, $s\in S$ and the
generators $S$ of $H$. In this section we derive a lower bound on the F{\o}lner function on $Sym_H$
in terms of the volume growth of $H$.

Let $Sym(X)$ be the symmetric group on a countable set $X$.
For a finite subset $Y\subset X$ denote by $T(Y)$ the set of transpositions between points in $Y$,
$T(Y)=\{(y_1,y_2):\ y_1,y_2\in Y\}$. Here $(y_1,y_2)$ denotes the permutation that transposes $y_1$ and $y_2$.
The following lemma provides a lower bound for the cardinality of 
sets in $Sym(X)$ which are $C$-satisfactory with respect to $T(Y)$. By definition, $V$ is $C$-satisfactory with respect to $T(Y)$
if for each $v\in V$ and at least $C\#T(Y)=\frac{C}{2}\#Y(\#Y-1)$ multiplications
by distinct transpositions $t\in T(Y)$ remain in $V$, that is, $vt\in V$.

\begin{lem}{[}Satisfactory sets in finite symmetric groups{]} \label{lem:symmetricgroup}
For each $C>0$, there exists $D=D(C)>0$ such that the following holds.
Let $Y$ be a finite subset of $X$. Suppose that a finite subset $V$ in $Sym(X)$
is $C$-satisfactory with respect to $T(Y)$,
then the cardinality of $V$ is at least $(D\#Y)^{D\#Y}$.

\end{lem}

\begin{proof} 
Write $n=\#Y$. We prove a more general claim. Let $m\in\N$,
and we say that $V\subset Sym(X)$ satisfies the assumption $(*)$ for $m$
if for each $v\in V$, there are at least $m$ distinct elements of $T(Y)$
such that $v$ multiplied with the element (on the right) remains
in $V$. We show that if $V$ verifies assumption $(*)$ for $m$, then the cardinality of
$V$ is at least $(m/(2n))^{m/(2n)}$.
The lemma is stated for $m=Cn^{2}$.

To prove this, observe that for any given $v\in V$, there exists a $y\in Y$, 
such that  $v(y,z)\in V$ for at least $m/n$ distinct $z$'s, $z\in Y$.
We fix one of such $v_0$ and $y$ and denote by $x=v_0^{-1}(y)$.
We subdivide $V$ according to the image of $x$: consider $V_{z}$ to be elements $v$ of $V$
such that $v(x)=z$. It is clear that $V$ is a disjoint union of $V_{k}$,
where the union is taken over $k\in X$.
Note that $\#V\ge \frac{m}{n}\min\#V_{z}$,
where the minimum is take over $z$ such $v_0(y,z)\in V$.

Next we show that for each $V_{z}$ where $z$ is such that $v_0(y,z)\in V$, the satisfactory assumption
$(*)$ is verified for $m'=m-n$.
Indeed, for $u\in V_{z}$ there are at least $m$ distinct transpositions in $T(Y)$ such that $ut\in V$.
Among them, if $t=(r,s)$ satisfies $ut\in V$ but $ut\notin V_z$, 
then it implies
that either $r=y$ or $s=y$. Changing if necessary the notation,
we can assume
that $r=y$. 
It follows that among these $m$ distinct transpositions, there are at most $n$ of them such that $ut\in V$ but $ut\notin V_z$.
In other words, for each $u\in V_z$, there are at least $m-n$ distinct transpositions $t\in T(Y)$ such that $ut\in V_z$. 

Repeat this process for $m/2n$ steps, we have that the cardinality of $V$ is 
at least 
$$\#V\ge \left(m/n\right)\left((m-n)/n\right)\left((m-2n)/n\right)\dots\ge(m/(2n))^{m/(2n)}.$$
In particular, for $V$ verifying the assumption $(*)$ for $m\ge Cn^{2}$,
the cardinality of $V$ is at least $Dn^{Dn}$, for some constant
$D>0$ depending only on $C$.

\end{proof}

We now proceed to prove Corollary \ref{cor:sym}. 
Given $h_{1},h_{2}\in H$, the transposition
$(h_{1},h_{2})$ has length at most $4l_{H,S_{H}}(h_{1})+2l_{H,S_{H}}(h_{2})$.
Indeed, $(h_1,h_2)=h_1^{-1}(e,h_2h_1^{-1})h_1$. Take a transposition $(e,h)$ where $e$ is the identity element.
Write $h$ as a shortest path in the generators: $h=s_{i_1}\cdots s_{i_{\ell}}$ where $\ell=l_{H,S}(h)$.
We have $$(e,h)=(e,s_{i_1})\cdots(e,s_{i_{\ell-1}})(e,s_{i_{\ell}})((e,s_{i_1})\cdots(e,s_{i_{\ell-1}}))^{-1}.$$
Therefore, $l_{Sym_H,\tilde{S}}((e,h))\le 2l_{H,S}(h)-1$. It follows that
$$l_{Sym_H,\tilde{S}}((h_1,h_2))\le 2l_{H,S}(h_1)+2l_{H,S}(h_2h_1^{-1})\le 4l_{H,S}(h_1)+2l_{H,S}(h_2).$$

\begin{proof}[Proof of Corollary \ref{cor:sym}]
Given $n\ge1$ consider
the set $T_{n}$ of transpositions 
$$T_n=T(B_H(e,n))=\{(h,h')|\ h,h'\in H, \mbox{where }l_{H,S}(h),l_{H,S}(h')\le n\}.$$
Then by the length estimate above, we have that elements of $T_n$ are of length at most $6n$ in $Sym_H$.

Let $V$ be a F{\o}lner set of $Sym_H$ such that $\#\partial_{\tilde{S}}V/\#V\le 1/200n$. Then by Theorem \ref{thm:satisfactorysets},
this set is $1/4$-satisfactory with respect to $T_n$. Since $T_n\subset Sym(H)$, we can assume, by taking cosets, that $V$ contains a subset 
$V'\subset Sym(H)$ that is $1/4$-satisfactory with respect to $T_n$. 
Apply Lemma \ref{lem:symmetricgroup} with $Y=B_{H,S}(e,n)$,  
it follows that there exists an absolute constant $c>0$ such that 
$$\#V\ge (cv_{H,S}(n))^{cv_{H,S}(n)}.$$

\end{proof}

Note that the lower bound of $\fol_{Sym_H}$ in Corollary \ref{cor:sym} is better 
than what we would obtain by applying Corollary \ref{thm:inequality}
instead of volume lower bound for satisfactory sets with respect to transpositions. 
Indeed, for any choice of abelian subgroup of $Sym(B_{H,S}(e,n))$,
by applying Corollary \ref{thm:inequality},
we only get the lower bound of order $\exp(n)$.  

\begin{exa}

Take $H=\Z^d$, by Corollary \ref{cor:sym}, the F{\o}lner function of 
$G=Sym(\Z^d)\rtimes \Z^d$ is asymptotically greater to equal to $n^{n^d}$. 
The group $G$ admits F{\o}lner pairs adapted to the function $n^{n^d}$.
By \cite{CGP} and the general relation between F{\o}lner function and decay of return probability
(see e.g. \cite[Theorem 14.3]{woessbook}), we deduce that the return probability 
$\mu^{(2n)}(e)$ of simple random walk on $Sym(\Z^d)\rtimes \Z^d$ is asymptotically equivalent to $\exp\left(-n^{d/d+2}\log^{2/d+2}n\right)$.
We mention that random walk on $Sym(H)\rtimes H$ is called the mixer chain in 
Yadin \cite{yadin}, where the drift function of the random walk on $Sym(\Z)\rtimes \Z$ 
is estimated.

\end{exa}

\section{Lacunary hyperbolic examples}\label{lacunary}

In this section we show that there exist non-virtually cyclic amenable
groups with Shalom's property $H_{\mathrm{FD}}$ that are lacunary hyperbolic. 
Asymptotic cones first appeared in the proof of the polynomial growth theorem by Gromov in \cite{Gromov2}. Roughly speaking, an asymptotic cone of a metric space is what one sees when the space is viewed from infinitely far away. For a definition 
using ultrafilters see van den Dries and Wilkie \cite{VDW}.
A well-known result of Gromov \cite{Gromov3} states that a finitely generated group is hyperbolic if and only if all its asymptotic cones are $\mathbb{R}$-trees. 
Recall that a group is lacunary
hyperbolic if one of its asymptotic cone is an $\mathbb{R}$-tree.
By Olshanskii, Osin and Sapir \cite[Theorem 1.1]{olshosinsapir}, lacunary hyperbolic groups can be characterized
as certain direct limits of hyperbolic groups. More precisely, a finitely
generated group $G=\left\langle S\right\rangle $ is lacunary hyperbolic
if and only if $G$ is the direct limit of a sequence of hyperbolic
groups $G_{i}=\left\langle S_{i}\right\rangle $, $S_{i}$ finite,
with epimorphisms 
\[
G_{1}\xrightarrow{\alpha_{1}}G_{2}\xrightarrow{\alpha_{2}}\ldots
\]
satisfying $\alpha_{i}(S_{i})=S_{i+1}$ and $\delta_{i}/r_i\to 0$ as $i\to\infty$,
where $\delta_{i}$ is the hyperbolicity constant of $G_{i}$ relative
to generating set $S_{i}$, and $r_i$ is injectivity radius of the map $\alpha_{i}$.

Osin, Olshanskii and Sapir in \cite{olshosinsapir} constructed lacunary hyperbolic groups  that served as examples/counter examples to various question. In particular, using central extension of lacunary hyperbolic groups they show that there exists a group whose asymptotic cone with countable but non-trivial fundamental group. They prove that what is called "divergence function"  (measuring how much the distance between two points outside the ball of given radius increases after removing this ball) can be arbitrary close to linear, but not being linear. They also show that the class of lacunary hyperbolic group is quite large: for example, one can find infinite torsion groups among lacunary hyperbolic groups; some of lacunary hyperbolic groups are amenable.
Properties of endomorphisms and automorphisms of lacunary hyperbolic groups were investigated in Coulon-Guirardel \cite{CouGui}. In particular, every lacunary hyperbolic group is Hopfian, see \cite[Theorem 4.3]{CouGui}.

We now briefly describe the construction of lacunary hyperbolic elementary
amenable groups in \cite[Section 3.5]{olshosinsapir}. These groups are locally-nilpotent-by-$\mathbb{Z}$, which was considered previously in \cite[Section 8]{gromoventropy}. 
Let $p$ be a prime number
and $\mathbf{c}=\left(c_{i}\right)$ be an increasing sequence of
positive integers. Let $A=A(p,\mathbf{c})$ be the group generated
by $b_{i}$, $b_{i}^{p}=1$, $i\in\mathbb{Z}$ with the following defining relations:
\[
\left[\left[b_{i_{0}},b_{i_{1}}\right],\ldots,b_{i_{c_{n}}}\right]=1\mbox{ if }\max_{j,k}\left|i_{j}-i_{k}\right|\le n.
\]
The group $A$ is locally nilpotent, and it admits an automorphism
$a_{i}\to a_{i+1}$. Let $G=G(p,\mathbf{c})$ be the extension of
$A$ by this automorphism, $G=\left\langle A,t\right\rangle $ where
$ta_{i}t^{-1}=a_{i+1}$. By \cite[Lemma 3.24]{olshosinsapir}, if the sequence
$\mathbf{c}$ grows fast enough, then $G$ is lacunary hyperbolic. 

In what follows, we consider a variation of the construction that introduces
a sequence of slow scales where the
group $G$ is close to a lamplighter, while on another sequence of
scales $G$ can be approximated by hyperbolic groups. In particular,
simple random walk on $G$ is cautious along a subsequence of time
$(t_{i})$
with appropriate choice of parameters.

\begin{thm}\label{lacunarycautious}

There exists a finitely generated lacunary hyperbolic non-virtually
cyclic amenable group $G$ with Shalom's property $H_{\mathrm{FD}}$. Moreover, the group
$G$ can be chosen to be locally-nilpotent-by-$\mathbb{Z}$ and simple random walk
on $G$ is cautious and diffusive along an infinite subsequence of time instances.

\end{thm}

\begin{proof}

Let $M=\ast_{j\in\mathbb{Z}}\left\langle b_{j}\right\rangle $, $b_{j}^{p}=1$
be the free product of copies of $\mathbb{Z}/p\mathbb{Z}$ indexed by $\mathbb{Z}$,
$\boldsymbol{\Gamma}=M\rtimes\mathbb{Z}$ be its cyclic extension
where $\mathbb{Z}$ acts by shifting indices. Let $\pi_{0}$ be the quotient map
from $\boldsymbol{\Gamma}\to \Z\wr(\mathbb{Z}/p\mathbb{Z})$.
Write $G_0=\Z\wr(\Z/p\Z)$.
We define a sequence
of quotients of $\Gamma$ recursively as follows (these quotients are
determined by a triple of parameters $(\ell_{i},c_{i},k_{i})_{i\in\mathbb{N}}$). 

After we have defined $G_i$, 
given parameter $\ell_{i+1}$, take the nilpotent subgroup in $G_i$ generated by $b_0,\ldots, b_{2^{\ell_{i+1}}-1}$,
$$N_{i+1}=\left\langle b_{0},\ldots,b_{2^{\ell_{i+1}}-1}\right\rangle.$$
Consider the quotient group $\hat{M}_{i+1}$ of $M$
defined by imposing the relations that for any $j$, the subgroup
generated by $\left\langle b_{j},\ldots,b_{j+2^{\ell_{i+1}}-1}\right\rangle $
is isomorphic to $N_{i+1}$ (isomorphism given by shifting indices
by $j$). Let $\hat{\Gamma}_{i+1}=\hat{M}_{i+1}\rtimes\mathbb{Z}$
be the cyclic extension of $\hat{M}_{i+1}$. Then $\hat{\Gamma}_{i+1}$
splits as an HNN-extension of the finite nilpotent group $N_{i+1}$, 
therefore $\hat{\Gamma}_{i+1}$ is virtually free, see \cite[Proposition 11]{serre}.
Therefore $\hat{\Gamma}_{i+1}$
is hyperbolic, let $\delta_{i+1}$ be the hyperbolicity constant of
$\hat{\Gamma}_{i+1}$ with respect to the generating set $S=\{b_{0},t\}$.
For parameters $c_{i+1},k_{i+1}\in\mathbb{N}$, consider the quotient
group $\bar{\Gamma}_{i+1}=\bar{\Gamma}_{i+1}(c_{i+1},k_{i+1})$ of $\hat{\Gamma}_{i+1}$
subject to additional relations $(\ast)$
\begin{align*}
\left[\left[\left[b_{j_{1}},b_{j_{2}}\right],b_{j_{3}}\right],\ldots,b_{j_{m}}\right] & =0\mbox{ for any }m\ge c_{i+1},\\
b_{j} & =b_{j+2^{k_{i+1}}}\mbox{ for all }j\in\mathbb{Z}.
\end{align*}
By choosing $c_{i+1}$ and $k_{i+1}$ to be sufficiently large, we
can guarantee that the injectivity radius $r_{i}$ of the quotient
map $\hat{\Gamma}_{i+1}\to\bar{\Gamma}_{i+1}$ satisfies $r_{i+1}\gg\delta_{i+1}$. 

Let $\psi_{i+1}:\boldsymbol{\Gamma}\to\bar{\Gamma}_{i+1}$ be the
projection from $\boldsymbol{\Gamma}$ to $\bar{\Gamma}_{i+1}$. Recall that $\pi_{0}$ denotes the projection 
from $\boldsymbol{\Gamma}\to Z\wr(\mathbb{Z}/p\mathbb{Z})$. We take
the group $G_{i+1}$ to be
\[
G_{i+1}=\boldsymbol{\Gamma}/\left(\ker\psi_{i}\cap\ker\pi_{0}\right).
\]

By construction, if we choose $k_{i+1},c_{i+1}\gg\ell_{i+1}$, $\ell_{i+1}\gg\max\left\{ \ell_{i},k_{i},c_{i}\right\} $
to be large enough parameters, we have that $G_{i+1}$ and $G_{i}$
coincide on the ball of radius $2^{\ell_{i+1}}$ around the identity
element, and $G_{i+1}$ coincide with a hyperbolic group $\hat{\Gamma}_{i+1}$
on the ball of radius $r_{i+1}$ such that $r_{i+1}\gg\delta_{i+1}$,
where $\delta_{i+1}$ is the hyperbolicity constant of $\hat{\Gamma}_{i+1}$.
Let $G$ be the limit of $(G_{i})$ in the Cayley topology, or equivalently,
\[
G=\boldsymbol{\Gamma}/\cap_{i\in\mathbb{N}}\ker\psi_{i}.
\]
Since along the sequence $(r_{i})$, the ball of radius $r_{i}$ around
identity in $G$ coincide with a ball of same radius in a hyperbolic
group with hyperbolicity constant $\delta_{i}\ll r_{i}$, it follows
that $G$ is lacunary hyperbolic. \\

We now show that for parameters $\ell_{i+1}\gg\max\left\{ \ell_{i},k_{i},c_{i}\right\} $
large enough, $G$ admits a subsequence of controlled F{\o}lner pairs, and simple random walk on $G$ is cautious and diffusive
along a subsequence. The argument is along the same line as in the proof of Theorem \ref{prescribe}.

Since the balls of radius $2^{\ell_{i+1}}$ are the same in $G$ and $G_i$, up to radius $2^{\ell_{i+1}}$ it is the same to consider the corresponding 
random walk in $G_i$. Let $H_i$ be the subgroup of $G_i$ generated by $b_0,\ldots, b_{2^{k_i}-1}$. Note that $H_i$ is a finite nilpotent group.
Then because of the relations $(\ast)$ in $\bar{\Gamma}_i$, 
$G_i$ fits into the exact sequence
$$1\to [H_i,H_i]\to G_i\to \Z\wr(\Z/p\Z)\to 1.$$
We use the same letter $\pi_0$ to denote the projection $G_i\to\Z\wr(\Z/p\Z)$ as well 
($\pi_0$ was used for the projection $\boldsymbol{\Gamma}\to \Z\wr(\Z/p\Z)$).
For $r\le 2^{\ell_{i+1}-C}$, let 
$$F'_r=\{g\in G_i:\ \pi_0(g)=(f,z), {\rm{supp}}f\subset [-r,r],\ |z|\le r\},$$
$$F_r=\{g\in G_i:\ \pi_0(g)=(f,z), {\rm{supp}}f\subset [-r,r],\ |z|\le r/2\}.$$
Then by definition, $(F_r,F'_r)$ forms a F{\o}lner pair. 
The outer set $F'_r$ has diameter bounded by
$${\rm{Diam}}_{G_i,S}\le {\rm{Diam}}_{G_i,S}([H_i,H_i])+10r.$$
Since the diameter of the finite subgroup $[H_i,H_i]$ in $G_i$ depends only on the choice of parameters up to index $i$, 
we have that for $r\ge r_0(c_i,k_i)\ge{\rm{Diam}}_{G_i,S}([H_i,H_i])$, the set $F'_r$ 
is contained in the ball $B(e,20r)$ in $G_i$. Thus for such $r$ sufficiently large, $(F_r,F'_r)$ is a controlled F{\o}lner pair in $G_i$. 
Since the balls of radius $2^{\ell_{i+1}}$ are the same in $G$ and $G_i$, for $\ell_{i+1}\gg c_i,k_i$ sufficiently large such that 
$r_0(c_i,k_i)\preceq 2^{\ell_{i+1}-C}$,
$(F_r,F'_r)$ can be identified as a controlled F{\o}lner pair in $G$ as well. 
Therefore if the sequence $(\ell_{i+1})$ grows sufficiently fast compared to $(c_i,k_i)$
in the sense described above,
we have that $G$ admits a subsequence of controlled F{\o}lner pairs.  
By Lemma \ref{pairs}, simple random walk on $G$ is cautious.

Similarly, for simple random walk on $G_i$, the drift function is bounded by
$$L_{G_i,\mu}(t)\le {\rm{Diam}}_{G_i,S}([H_i,H_i])+L_{G_0,\mu}(n),$$
where $G_0=\Z\wr(\Z/p\Z)$. 
Thus when $r_0(c_i,k_i)\preceq 2^{\ell_{i+1}-C}$, there exists some constant $\ell_{i+1}'$ 
depending on $c_i,k_i$, such that for $t\in [2^{\ell'_{i+1}},2^{\ell_{i+1}}]$, the constructed group $G_i$ satisfies
$$L_{G,\mu}(t)=L_{G_i,\mu}(t)\le C\sqrt{t}.$$
That is, simple random walk on $G$ is diffusive 

\end{proof}

\begin{rem}

From the construction of the group $G$, we have that along a sequence of radius $(r_i)$, the ball of radius $r_i$ around identity of $G$ coincide with the ball of the same radius in a virtually free group. It follows that the drift function can be close to linear along a subsequence of time $(t_i)$.  

\end{rem}

\begin{rem}
The fact that the lamplighter group $\mathbb{Z}\wr(\mathbb{Z}/2\mathbb{Z})$ can be realized as a direct limit of virtually free group with growing injectivity radius was used by Osin in \cite{osinKazhdan} to show that the Kazhdan constant of a hyperbolic group is not bounded uniformly from below under changing generating sets.
\end{rem}

It is known that elementary amenable groups can have arbitrarily fast F{\o}lner functions, see \cite[Corollary 1.5]{olshosin} (also the remarks in \cite[Section3]{erschlerfoelner} and \cite[Section8]{gromoventropy}).

\begin{rem}
By taking the direct product of two lacunary hyperbolic groups as in Theorem \ref{lacunarycautious}, we obtain locally-nilpotent-by-abelian group
with property $H_{\mathrm{FD}}$ for which it can be shown (similar to
\cite[Section8]{gromoventropy}) that the F{\o}lner function is arbitrarily fast. 
In addition to groups in \cite{brieusselzheng}, they can provide another collection of elementary amenable groups with arbitrarily fast F{\o}lner functions 
while simple random walks on them have trivial Poisson boundary. Amenable (but non-elementary amenable)groups of with this property were constructed in \cite{erschlerpiecewise}.
\end{rem}

\textsc{\newline Anna Erschler --- 
D\'{e}partement de math\'{e}matiques et applications, \'{E}cole normale
sup\'{e}rieure, CNRS, PSL Research University, 45 rue d'Ulm, 75005 Paris
} --- anna.erschler@ens.fr

\textsc{\newline Tianyi Zheng --- Department of Mathematics, UC San Diego,
9500 Giman Dr. La Jolla, CA 92093
} --- tzheng2@math.ucsd.edu

\end{document}